\documentclass[12pt]{amsart}
\usepackage{amsmath,amssymb}
\usepackage[english]{babel}
\usepackage[colorlinks,
linkcolor=blue,
urlcolor=blue,
citecolor=red]{hyperref}
\usepackage[usenames,dvipsnames]{xcolor}
\usepackage{microtype}
\allowdisplaybreaks

\newcommand{\R}{{\mathbf R}}
\newcommand{\RR}{{\mathcal R}}

\newcommand{\Len}{\mathcal{L}}

\newcommand{\spt}{\mathop{\rm spt}}

\renewcommand{\b}{{\rm b}}
\newcommand{\p}{{\partial}}

\newcommand{\Hess}{\mathop{\rm Hess}}

\newcommand{\Rc}{{\rm Ric}}

\newcommand{\T}{{\Lambda}}

\newcommand{\vol}{{\rm vol}}

\newcommand{\lims}{\varlimsup}

\renewcommand{\d}{{d}}

\newtheorem{theorem}{Theorem}

\newtheorem{corollary}[theorem]{Corollary}

\newtheorem{definition}[theorem]{Definition}

\newtheorem{lemma}[theorem]{Lemma}

\newtheorem{proposition}[theorem]{Proposition}
\newtheorem{remark}[theorem]{Remark}

\newdimen\dummy
\dummy=\oddsidemargin
\newcommand{\fr}{\penalty-20\null\hfill$\blacksquare$}         %quadratino nero alla fine del remark

\DeclareMathOperator{\supp}{\spt}
\newcommand{\Dis}{{\mathcal{D}}}
\newcommand{\Vol}{\mathrm{Vol}}

\newcommand{\eps}{\varepsilon}

\usepackage{todonotes}

\begin{document}

\title[Weighted $C^1$-Lorentzian splitting theorem]
{A Lorentzian splitting theorem for continuously differentiable 
metrics and weights 
}
\author[Braun]{Mathias Braun}
\address{Institute of Mathematics, EPFL, 1015 Lausanne, Switzerland \tt mathias. braun@epfl.ch}
\author[Gigli]{Nicola Gigli}\address{SISSA, Trieste, Italy \tt ngigli@sissa.it}
\author[McCann]{Robert J. McCann}
\address{Department of Mathematics, University of Toronto, Toronto, Ontario, Canada \tt mccann@math.toronto.edu}
\author[Ohanyan]{Argam Ohanyan}
\address{Department of Mathematics, University of Toronto, Toronto, Ontario, Canada \tt argam.ohanyan@utoronto.ca}
\author[S\"amann]{Clemens S\"amann}
\address{{Faculty of Mathematics, University of Vienna, Vienna, Austria\newline \tt clemens.saemann@univie.ac.at}}

\thanks{2020 MSC Classification Primary: 83C75 Secondary: 35J92 35Q75 49Q22 51K10 53C21 53C50 58J05.}

\thanks{\copyright \today\ by the authors.}
\begin{abstract}
We prove a splitting theorem  for globally hyperbolic, weighted spacetimes with metrics and weights of regularity $C^1$ by combining elliptic techniques for the negative homogeneity $p$-d'Alembert operator from our recent work in the smooth setting with the concept of line-adapted curves introduced here. Our results extend the Lorentzian splitting theorem proved for smooth globally hyperbolic spacetimes by Galloway --- and variants of its weighted counterparts by Case and Woolgar--Wylie --- to this low regularity setting.
\end{abstract}

\maketitle

\tableofcontents

\section{Introduction}

Einstein's theory of gravity postulates that the geometry of spacetime is described by a Lorentzian metric $g$ on a smooth manifold $M$ that solves
the Einstein field equation:  a second-order nonlinear system relating the Ricci curvature of the metric tensor  
to the physics described by the total densities and fluxes of energy and momentum in the system.  Even for the vacuum Einstein equation, solutions need not generally be smooth.  On one hand,  the works of Penrose and Hawking give generic conditions under which 
geodesics become incomplete either inside black hole horizons (see Penrose \cite{Penrose65a}) or analogous to the big bang (see Hawking \cite{Hawking66a}).  On the other hand, linearization of the field equations yields a wave equation,  and wave equations are well-known to support (and to propagate) nonsmooth solutions --- in this case representing nonsmooth gravity waves.  It is thus highly desirable for a theory of gravity which encompasses nonsmooth metric tensors $g$.

Nonsmooth proofs of the Hawking and Penrose theorems
have been provided for metric of class $g \in C^1$ by Graf \cite{Graf20}; see also Kunzinger--Ohanyan--Schinnerl--Steinbauer \cite{KunzingerOhanyanSchinnerlSteinbauer2022} for a version of the more sophisticated Hawking--Penrose theorem in $C^1$-regularity, as well as Schinnerl--Steinbauer \cite{SchinnerlSteinbauer} for analogous work on the Gannon--Lee theorems. More recently, Calisti--Graf--Hafemann--Kunzinger--Steinbauer \cite{calisti2025hawking} extended the Hawking singularity theorem to locally Lipschitz metrics.  Versions of the Hawking (and Penrose) 
theorems have also been obtained by Cavalletti and Mondino~\cite{CavallettiMondino20+}
(with Manini~\cite{CavallettiManiniMondino25+}) by building on the framework of Kunzinger and S\"amann~\cite{KunzingerSaemann18},
which relaxes even the assumption that the underlying spacetime $M$ is a manifold.

In our previous work \cite{BGMOS24+} we gave an elliptic proof of the Eschenburg~\cite{Eschenburg88}, Galloway \cite{Galloway89} and Newman \cite{Newman90} splitting theorems which apply to smooth Lorentzian metrics. {Central to our approach was the $p$-d'Alembert operator $\square_p$ for $0 \neq p < 1$, whose elliptic action on non-smooth time functions was already observed 
in} \cite{BeranOctet24+}; cf.\ \cite{braun2024nonsmooth}. In the present manuscript, we combine this novel strategy with additional ideas to obtain a comparable but new result for metric tensors which are merely of regularity
$g \in C^{1}$, and for which the Ricci tensor can moreover be modified to include a $C^1$-smooth Bakry--\'Emery style weight~\cite{BakryEmery85},  as in e.g.\ Case \cite{Case10} and
Woolgar \cite{Woolgar13} with Wylie \cite{WoolgarWylie16, WoolgarWylie18}. Our nonsmooth result extends Galloway's version of the theorem \cite{Galloway89}, which requires global hyperbolicity. In the smooth setting either global hyperbolicity or timelike geodesic completeness suffices, as shown by Galloway \cite{Galloway89} and by Newman \cite{Newman90} respectively, the latter having been conjectured a decade earlier by Yau \cite{Yau82}.  
Global hyperbolicity is widely acknowledged to be the more physically relevant completeness assumption, as in e.g.\ Witten \cite{Witten20}.  For another relaxation of the splitting hypotheses see \cite{galloway2025note}.

Let us recall some basics about spacetimes with metrics of regularity $C^1$. A pair $(M,g)$ will be called a {\bf $C^1$-spacetime} if $M$ is a  {second countable, connected, smooth manifold, equipped with  
a signature $(+,-,\ldots,-)$ metric tensor $g$ of regularity $C^1$, and there exists a continuous vector field $F$ on $M$ satisfying $g(F,F)>0$ to} distinguish future from past. 
For $g_{ij} \in C^1(M)\backslash C^{1,1}_{loc}(M)$, geodesics still exist (by Peano's theorem) but need no longer be determined by their initial velocities (since the right hand side of the geodesic equation 
\begin{equation}\label{geodesic eqn}
    \frac{d^2 \gamma^i}{d t^2} = - \Gamma^{i}_{jk} \frac{d \gamma^j}{dt}\frac{d \gamma^k}{dt} \qquad \forall i \in \{ 1,\ldots,n=\dim M\}
\end{equation}
is merely $C^0$). However, as in Graf \cite[Cor.\ 2.4]{Graf20}, geodesics in this setting remain $C^2$-regular and causal geodesics do not change their causal {type}. Here, a locally Lipschitz curve $\gamma$ is called {\bf causal} if $g(\gamma'_t,\gamma'_t) \geq 0$ for a.e.\ $t  \in I$. The Lorentzian arclength functional $\Len$ is defined in the usual way for locally Lipschitz causal curves $\gamma:I \subset \R \to M$:
\begin{equation*}
    \Len(\gamma) := \int_I \sqrt{g(\gamma'_t,\gamma'_t)} \, dt.
\end{equation*}
A causal curve is said to be {\bf future-directed} --- or {\bf f.d.}\ for short ---
if in addition $g(\gamma'_t,F)\ge 0$ for a.e.\ $t \in I$.
The time separation is then
\begin{equation}
\label{time separation}
    \ell(x,y):=\sup\{\Len(\gamma) : \gamma:[a,b] \to M \text{ f.d.}, \gamma_a = x, \gamma_b = y\},
\end{equation}
where the supremum is understood to be $-\infty$ if there are no f.d.\ curves $\gamma$ from $x$ to $y$. Curves achieving the maximum \eqref{time separation} 
are called {\bf maximizing}, {\bf maximizers} or {\bf maximal};
 they are necessarily (reparametrizations of) geodesics in the ODE sense, i.e., solutions of the geodesic equation \eqref{geodesic eqn}, see Lange--Lytchak--S\"amann \cite{LangeLytchakSaemann}. Unless noted otherwise, by {\bf geodesic} we will always mean a solution of the geodesic equation, hence in particular affinely parametrized.

A $C^1$-spacetime $(M,g)$ is called {\bf globally hyperbolic} if it is non-totally imprisoning (i.e.\ no compact set contains a future or past inextendible causal curve) and the causal diamonds $J^+(x) \cap J^-(z)$ are compact for each $x,z \in M$.  Given $-\infty \le a < b \le \infty$,  a causal curve $\gamma:(a,b) \longrightarrow M$ is said to be 
{\bf inextendible} unless
$$
\lim_{t \downarrow a} \sigma_t
\quad {\rm or} \quad
\lim_{t \uparrow b} \sigma_t
$$
converges to a limit in the (boundaryless) manifold $M$.
Standard consequences of global hyperbolicity which remain true
for $g \in C^1$ include the attainment of \eqref{time separation}
whenever $\ell(x,y) \ge 0$, the 
upper semicontinuity of the time separation $\ell$, and continuity and real-valuedness of $\ell_+ := \max\{\ell,0\}$; see S\"amann \cite{Saemann16}.

Given $V \in C^1(M)$
and a parameter $N \in [n,\infty)$
--- with $n:=\dim M$ and the convention $V=0$ if $N=n$ ---
the Bakry--\'Emery Ricci tensor can be defined as the distributional analog  of
\begin{equation}\label{B-E Ric}
\Rc^{(g,N,V)} := \Rc + \Hess V - \frac{1}{N-n}dV \otimes dV,
\end{equation}
noting the fact that 
$\smash{C^{1}(M) \subset W^{1,2}_{loc}(M)}$ makes the metric tensor a fortiori Geroch--Traschen class \cite{GerochTraschen87}.

\begin{definition}[Bakry--\'Emery energy condition; timelike distributionally]\label{D:BEC}
We say the {\bf Bakry--\'Emery energy condition} is satisfied if
$\Rc^{(g,N,V)} \ge 0$ holds {\bf timelike distributionally}, i.e.\ whenever contracted with $X \otimes X$ and integrated against $\mu$,  for each smooth (or, equivalently, $C^1$) timelike vector field $X$ 
and each smooth compactly supported (unoriented) volume density $\mu \ge 0$ on $M$; cf. Graf \cite[Definition 3.3]{Graf20}.
\end{definition}

For us,  a {\bf complete timelike line} is a curve $\gamma:\R\to M$ such that
\begin{equation}
    \label{eq:defline}
    \ell(\gamma_s,\gamma_t)=t-s \qquad\forall s,t\in\R,\ s<t.
\end{equation}
Such lines are known to be locally Lipschitz maximizing $g$-geodesics~\cite{McCann24} --- thus in fact $C^2$-smooth \cite{LangeLytchakSaemann};  cf. Lemma \ref{le:regularityadapted} below.

\begin{theorem}[Weighted nonsmooth Lorentzian splitting theorem]
\label{T2:weighted splitting}
Let $(M,g)$ be an $n$-dimensional globally hyperbolic spacetime with $C^1$ metric tensor,  $N \in [n,\infty)$ and $V \in C^1(M)$. Assume   
the Bakry--\'Emery energy condition, {i.e.\ $\Rc^{(g,N,V)} \geq 0$} is satisfied {timelike distributionally} as in Definition 
\ref{D:BEC}, and  that
$M$ contains a complete timelike line $\gamma:\R \to M$. Then there is a complete Riemannian manifold $(S,h)$ with $C^1$ metric tensor and {a bijection} $\Phi:\R\times S\to M$ so that:
\begin{itemize}
    \item[-] $\Phi$ is a $C^2$ isometry,     and $g = dt^2-h$ on $M$;
    \item[-] $\Phi(t,z)=\gamma_t$ for every $t\in\R$ and some $z\in S$;
    \item[-] the weighted measure $e^{-V}\vol_g$ on $M$ splits via $\Phi$ as $dt \otimes e^{-V} \vol_h$; 
    \item[-] $(S,h)$ has a distributionally
nonnegative Bakry--\'Emery 
Ricci tensor $\Rc^{(h,N-1,V)}$.
\end{itemize}
\end{theorem}

\begin{remark}[Smoother splittings of smoother metrics]\label{R:weighted splitting}
{\rm With the notation and hypotheses of Theorem  \ref{T2:weighted splitting}, 
if $g$ is $C^k$ for some integer $k \ge 1$,
then the isometry $\Phi$ acquires regularity $\smash{C^{k+1}}$.  This follows from results of Hartman~\cite{Hartman58L}
 in Corollary~\ref{C: semiRiemannisometriesregularity} below. Alternately, it follows from the main result of Calabi and Hartman \cite{CalabiHartman1970},
after noting that $\Phi$ gives an isometry between the Riemannian metrics $dt^2+h$ on $\R \times S$ and $2 dt^2-g$ on $M$.
Similarly, if $g \in C^{k,\alpha}$ for $k \ge 1$ and $\alpha \in (0,1)$,  then $\Phi \in C^{k+1,\alpha}$ by Taylor's results for Riemannian isometries \cite[Thm. 2.1]{taylor2006existence}. 
}\fr\end{remark}

\begin{remark}[On the range of synthetic dimension $N$]{\rm 
In the context of smooth Lorentzian splitting theorems, there are works where the cases $N=\infty$ and even $N < 0$ are permitted, see Case \cite{Case10} and Woolgar--Wylie \cite{WoolgarWylie16, WoolgarWylie18} (and Lu--Minguzzi--Ohta \cite{LuMinguzziOhta23} as well as Caponio--Ohanyan--Ohta \cite{CaponioOhanyanOhta24} in the Lorentz--Finsler setting). In these ranges of $N$, a further completeness assumption is required to ensure the weighted $p$-sub- and  
$p$-superharmonicity, respectively, of the Busemann function in the limit of the $p$-d'Alembertian comparisons obtained for the approximate Busemann functions. In order not to distract from the adaptation of our $p$-d'Alembertian techniques from \cite{BGMOS24+} to the low regularity setting, we will not consider these ranges of $N$ in this work.}\fr
\end{remark}

{Relative to previous approaches surveyed in our prequel \cite{BGMOS24+}, 
a technical innovation of the present manuscript
--- inspired by optimal transportation ideas \cite{Ambrosio03,CaffarelliFeldmanMcCann00,TrudingerWang01} --- 
is the early introduction of curves $\sigma$ which are {\bf $\gamma$-adapted},  meaning proper-time parametrized curves $\sigma:(a,b) \longrightarrow M $ along which
the forward and backward Busemann functions $\b^\pm$
associated to the complete timelike line $\gamma$ --- as in \eqref{approximate bi-Busemann}--\eqref{bi-Busemann limit} below --- agree 
and increase at their minimum possible (i.e.\ unit) rate.
In particular,  $\gamma$ itself is $\gamma$-adapted.  More generally, any $\gamma$-adapted curve $\sigma$ is a maximizing timelike geodesic.  
We prove that any co-rays of $\gamma$ (defined in Lemma~\ref{le:corays}) which start on $\sigma$
must be tangent to $\sigma$ and are hence necessarily timelike (Proposition~\ref{Proposition: generalizedcoraysalongrayagreewithray}). We exploit this property as a substitute for  the timelike co-ray condition of Beem, Ehrlich, Markvorsen and Galloway, or rather to enforce that the latter
holds not only near $\gamma$ but also near $\sigma$. 
As Eschenburg showed, where all co-rays are timelike  the forward and backward Busemann functions $\b^\pm$ 
are not only continuous~\cite{BeemEhrlichMarkvorsenGalloway85}
but locally Lipschitz~\cite{Eschenburg88}.  As in our  prequel~\cite{BGMOS24+},  this uniformizes the ellipticity of the $p$-d'Alembert operators acting on $\b^\pm$ near $\sigma$.
Relaxing the equi-semiconcavity from~\cite{BGMOS24+},  we show
the equi-superdifferentiability possessed by the Busemann limits in this less smooth setting (Proposition \ref{Pr:Equisupdist} below)
still allows us to deduce that $\b^+\ge \b^-$ are super- and sub-$p$-harmonic functions (Corollary \ref{L2:harmonicity}) which touch along $\sigma$.
The strong maximum principle (Proposition \ref{Proposition: Strongtangencyprinciple})
extends the equality $\b^+=\b^-$ to a neighbourhood $U$ of any $\gamma$-adapted curve, forcing $\b^+\in C^1(U)$.  
The fact that the Hessian of $\b:=\b^+=\b^-$ vanishes on $U$ 
(Proposition \ref{Proposition: higherregularity, weightconstant}) then follows from a weighted version  of the Bochner--Ohta  \cite{Bochner46} \cite{Ohta14h} identity as in \cite{BGMOS24+} (cf. \cite{MondinoSuhr23}) with a whiff of elliptic regularity theory (Lemma \ref{le:evans}).
 
 Having now established 
 $\b\in C^2(U)$, we are prepared to prove our splitting result,  namely that the metric splits into tangential and normal components along the level sets of $\b$
 (since $\nabla \b$ is parallel);  these are isometric to each other (since $\nabla \b$ is Killing).  We obtain this result not only in a neighbourhood of $\gamma$,
 but along every curve $\sigma$ that is $\gamma$-adapted.   We show that any co-rays (and asymptotes) of $\gamma$ intersecting flow lines of $\nabla \b$ are $\gamma$-adapted.
 This allows us to deduce they cannot end in finite-time using an  argument which seems simpler to us than Galloway's~\cite{Galloway89} --- and without recourse to the timelike geodesic completeness of \cite{Eschenburg88} \cite{Newman90} or the timelike nonbranching hypotheses of \cite{CavallettiMondino20+} \cite{KunzingerOhanyanSchinnerlSteinbauer2022}.
 This in turn implies that the asymptotes $\sigma$ to $\gamma$ form complete timelike lines,
 allowing us to conclude that the maximal neighbourhood $U$ around $\gamma$ which splits
 has a product structure (Theorem \ref{T:productstructure}), and that it must fill the manifold $M$ using a connectedness argument which seems simpler to us than Eschenburg's \cite{Eschenburg88} --- not to speak of its adaptations \cite{Galloway89}\cite{BGMOS24+}
 to the absence of the timelike geodesic completeness assumption.
The isometry provided is a priori $C^1$;  we upgrade it to $C^2$ in an appendix devoted to arbitrary isometries of semi-Riemannian metrics, or alternately using \cite{CalabiHartman1970} or \cite{taylor2006existence}. Finally, we close with an outlook on related directions of research to be explored in future work (Section \ref{Section: Conclusion}).

\section{Tensor distributions and smooth approximations} \label{Section: tensordistrib}

In this section, we summarize basics of distributional curvature bounds and optional weights. All results are by now well-established in the literature; we refer e.g.~to Kunzinger et al.\ \cite{KunzingerSteinbauerStojkovicVickers2015,KunzingerOhanyanSchinnerlSteinbauer2022}, Graf \cite{Graf20} and Braun--Calisti \cite{BraunCalisti2024} for more details. These sources should be consulted for technical details omitted here. A reader who is impatient or unfamiliar with tensor distributions can also skip directly to Lemma~\ref{Le:Good approx} below and take its conclusions as a shortcut \emph{definition} of ``Bakry--\'Emery Ricci curvature being bounded from below timelike distributionally.''

The intervening discussion will also clarify our $C^1$-hypotheses on the metric tensor and the weight, respectively.

In the sequel, we follow closely the presentation in \cite[Sec.\ 2]{KunzingerOhanyanSchinnerlSteinbauer2022}. For $k \in \mathbf{N}_0 \cup \{+\infty\}$, let $\Dis^{'(k)}(M)$ denote the topological dual of the locally convex topological vector space of all $k$-times continuously differentiable, compactly supported sections of the volume bundle $\Vol(M)$ of $M$ (such sections are also called  \emph{volume densities}; they do not carry orientations). We call $\Dis^{'(k)}(M)$ the space of \emph{distributions of order $k$} on $M$. If $k=\infty$, we simply write $\mathcal{D}'(M)$ and call it the space of \emph{distributions} on $M$. There are natural topological embeddings $\mathcal{D}^{'(k)}(M) \hookrightarrow \mathcal{D}^{'(k+1)}(M) \hookrightarrow \mathcal{D}'(M)$ for all $k$. Moreover, observe that $\mathcal{D}^{'(k)}(M)$ is a module over $C^k(M)$. (Recall that a module over a ring is an analogous structure to a vector space over a field; we denote the tensor product of such modules by $\otimes_{C^k(M)}$.) We call $u\in \Dis^{'(k)}(M)$ \emph{nonnegative} --- and write $u\geq 0$ --- if  $u(\mu)  \geq 0$ for every $k$-times continuously differentiable and compactly supported volume density $\mu \ge 0$. Correspondingly, given $v\in\Dis^{'(k)}(M)$ we write $u\geq v$ if $u-v$ is nonnegative.
Having defined (scalar) distributions on $M$, tensor distributions can now be obtained via tensor products: The space of \emph{tensor distributions of valence $(r,s)$ and order $k$} is defined by $\mathcal{D}^{'(k)}\mathcal{T}^r_s(M):= \mathcal{D}^{'(k)}(M) \otimes_{C^k(M)} \mathcal{T}^r_s(M)$, where $\mathcal{T}^r_s(M)$ are the smooth tensors on $M$ of valence $(r,s)$. Standard terminology is used for the more common valences, e.g.\ $\mathcal{D}'\mathcal{T}^1_0(M)$ are the \emph{distributional vector fields} on $M$; ($\mathfrak{X}(M):=\mathcal{T}^1_0(M)$ denotes the smooth vector fields).

A \emph{distributional connection} on a smooth manifold $M$ is a map $\nabla:~\mathfrak{X}(M) \times \mathfrak{X}(M) \to \mathcal{D}'\mathcal{T}^1_0(M)$ that satisfies the usual computation rules for connections. If $\nabla_X Y$ is an $L^p_{loc}$ (resp.\ $C^k$) vector field for all $X,Y \in \mathfrak{X}(M)$, then we call $\nabla$ an $L^p_{loc}$ (resp.\ $C^k$) connection. Connections which map into these higher regularity spaces can be extended to accommodate less regular vector fields $X,Y$ to give meaning to $\nabla_X Y$, see the discussion in \cite[p.\ 1149]{KunzingerOhanyanSchinnerlSteinbauer2022}. Using this, one can then define the \emph{distributional Riemann curvature tensor} $R \in \mathcal{D}' \mathcal{T}^1_3(M)$ of an $L^2_{loc}$-connection $\nabla$ (see \cite[Def.\ 3.3]{LeflochMardare} or \cite[Def.\ 2.2]{KunzingerOhanyanSchinnerlSteinbauer2022}) by the usual formula $R(X,Y)Z:=\nabla_X \nabla_Y Z - \nabla_Y \nabla_X Z - \nabla_{[X,Y]}Z$ (let us note here that the tensor distribution $R \in \mathcal{D}'\mathcal{T}^1_3(M)$ is understood as a $C^\infty(M)$-linear map $\mathfrak{X}(M)^3 \to \mathcal{D}' \mathcal{T}^1_0(M)$, in analogy with smooth tensors).

Let us turn to the setting of our main interest: If $g$ is a semi-Riemannian metric of regularity $C^1$ on a smooth manifold $M$, then its Levi-Civita connection $\nabla$ is a $C^0$-connection, from which it follows that $R \in \mathcal{D}^{'(1)}\mathcal{T}^1_3(M)$ \cite[p.\ 1150]{KunzingerOhanyanSchinnerlSteinbauer2022}. Due to this higher regularity, it was observed in the same reference that the object $R(X,Y,Z,W) = g(W,R(X,Y)Z)$ can be defined as an element of $\mathcal{D}^{'(1)}(M)$, even for $C^1$-vector fields $X,Y,Z,W$. This allows traces to be taken with respect to local $g$-orthonormal frames and coframes, which are of regularity $C^1$. Given an open set $U$, a $g$-orthonormal frame $E_i$ on $U$, then for all $C^1$-vector fields $X,Y$ on $U$, we define
\begin{align*}
    \Rc^g(X,Y):=\sum_{i=1}^{\dim M} g(E_i,E_i) g(R(E_i,X)Y,E_i).
\end{align*}
A partition of unity argument extends this definition to all $C^1$-vector fields $X,Y$ on $M$, from which the (global) Ricci curvature tensor $\Rc \in \mathcal{D}^{'(1)}\mathcal{T}^0_2(M)$ is obtained. 
Since we will consider weighted Ricci curvature bounds in this work, let us note that if $V \in C^1(M)$, then its distributional $g$-Hessian $\Hess V \in \mathcal{D}^{'(1)}\mathcal{T}^0_2(M)$ can be obtained following the analogous steps as for the Ricci curvature.

Given any local coordinate chart $(U,\varphi = (x^i))$ on a smooth manifold $M$ and $u \in \mathcal{D}^{'(k)}(M)$, the pushforward of $u|_U$, $\varphi_*u \in \mathcal{D}^{'(k)}(\varphi(U))$ is an ordinary $k$-th order distribution on $\varphi(U) \subset \R^n$, acting on a function $f \in C^k_c(U)$ via
\begin{align*}
    \varphi_*u(f) := u(f \circ \varphi \, |dx^1 \wedge \dots \wedge dx^n|).
\end{align*}
Similarly, if $T \in \mathcal{D}^{'(k)}\mathcal{T}^r_s(M)$, then $\varphi_*T$ is a tensor distribution on $\varphi(U)$ which can be written as
\begin{align*}
    \varphi_* T = T^{k_1\dots k_r}_{l_1\dots l_s} \frac{\partial}{\partial x^{k_1}} \otimes \dots \otimes \frac{\partial}{\partial x^{k_r}} \otimes dx^{l_1} \otimes \dots \otimes dx^{l_s},
\end{align*}
with $T^{k_1\dots k_r}_{l_1\dots l_s} \in \mathcal{D}^{'(k)}(\varphi(U))$.

For later use, let us recall the coordinate expressions of the Christoffel symbols, the Ricci curvature and the Hessian, obtained from a $C^1$-semi Riemannian metric $g$; these are simply the usual coordinate expressions which continue to be valid in this distributional regularity.
\begin{itemize}
\item $C^0$-Christoffel symbols:
\begin{equation}\label{Christoffel}
\smash{\Gamma_{ij}^k} = \frac12 g^{km}\left(\frac{\p g_{mj}}{\p x^i} + \frac{\p g_{im}}{\p x^j} -\frac{\p g_{ij}}{\p x^m}\right),
\end{equation}
\item Distributional Ricci curvature:
\begin{align}\label{Eq:RRRRRR}
\Rc^g_{ij} := \frac{\partial \Gamma_{ij}^m}{\partial x^m} - \frac{\partial\Gamma_{im}^m}{\partial x^j} + \Gamma_{ij}^m\,\Gamma_{km}^k - \Gamma_{ik}^m\,\Gamma_{jm}^k,
\end{align}
\item Distributional Hessian of $V \in C^1(M)$:
\begin{align}\label{Eq:HESS}
(\Hess V)_{ij} = \frac{\partial^2V}{\partial x^i\partial x^j} -\Gamma_{ij}^k\,\frac{\partial V}{\partial x^k}.
\end{align}
\end{itemize}

The distributional Bakry--\'Emery Ricci curvature $\Rc^{(g,N,V)}$, where $N\in [n,\infty)$, is now simply defined by \eqref{B-E Ric}. Here, we will use the conventions $\Rc^{(g,\infty,V)} := \Rc^g + \Hess\,V$ and $\Rc^{(g,n,V)} := \Rc^g$.

From now on let $g \in C^1$ be Lorentzian, i.e., of signature $(+,-,\dots,-)$. Given $K\in\R$ we say $\Rc^g \geq K$ (or more precisely $\Rc^g\geq K\,g$) timelike distributionally if the distribution $\Rc^g(X,X) -  K\,g(X,X) \in \mathcal{D}^{'(1)}(M)$ is nonnegative for every smooth (or equivalently $C^1$) timelike vector field $X$ on $M$. Analogous notions are adopted for the other curvature quantities. 

Let $(\rho_\varepsilon)_{\varepsilon > 0}$ be a family of  standard mollifiers in $\R^n$. By chartwise convolution and using  a partition of unity,  one can construct a family denoted by $\{g \star \rho_\varepsilon : \varepsilon > 0\}$ of $C^\infty$ Lorentzian metrics such that $g_\varepsilon \to g$ in $C^1_{loc}$ as $\varepsilon\to 0$ in a natural way, see e.g.\ \cite[Subsec.\ 2.2]{KunzingerOhanyanSchinnerlSteinbauer2022}. Similarly, one can convolve the distributional Bakry--\'Emery Ricci curvature $\smash{\Rc^{(g,N,V)}}$, which yields a family $\smash{\{\Rc^{(g,N,V)}\star\rho_\varepsilon : \varepsilon > 0\}}$ of smooth sections of the bundle $T^0_2M$ of valence $(0,2)$-tensors. As convolution preserves nonnegativity, this regularization process preserves lower bounds on the Bakry--\'Emery Ricci curvature in all timelike directions in the following way:
\begin{align*}
&\Rc^{(g,N,V)} \geq K\,g\quad\textnormal{timelike distributionally}\\
&\qquad\qquad \Longrightarrow\quad  \Rc^{(g,N,V)}\star \rho_\varepsilon \geq K\,g\star\rho_\varepsilon\quad\textnormal{pointwise}.
\end{align*}
To derive geometric statements from this regularization procedure, one would like to compare $\smash{\Rc^{(g,N,V)}\star \rho_\varepsilon}$ to $\smash{\Rc^{(g\star\rho_\varepsilon,N,V\star \rho_\varepsilon)}}$. Due to the non\-linear dependence of the Bakry--\'Emery Ricci tensor \eqref{Eq:RRRRRR} both on $g$ and $V$, these two quantities do not coincide in general. On the other hand, their error tends to zero locally uniformly: The case $V=0$ was established by Graf \cite[Lem.\ 4.5]{Graf20}, and her argument is easily seen to generalize to $\Rc^{(g,N,V)}$ since $V \in C^1(M)$.

The proposed approximation of $g$ will not respect the causal structure defined by $g$, which necessitates a modification of $g\star\rho_\varepsilon$ to a smooth Lorentzian metric $g_\varepsilon$ whose causal structure is comparable to the one of $g$. This was done by Kunzinger--Steinbauer--Stojkovi\'c--Vickers \cite{KunzingerSteinbauerStojkovicVickers2015} in regularity $g\in C^{1,1}$ and by Graf \cite{Graf20} in regularity $g\in C^1$ (following earlier work of Chru\'sciel--Grant \cite{ChruscielGrant}). We report their main result adapted to our weighted setting in Lemma~\ref{Le:Good approx} below (see also Braun--Calisti \cite{BraunCalisti2024}).

To formulate it, we  define $V_\varepsilon := V \star\rho_\varepsilon$ (understood with respect to the same partition of unity used to construct $g\star\rho_\varepsilon$). Given two not necessarily smooth Lorentzian metrics $g$ and $g'$, we write $g'\prec g$ if every $g'$-causal tangent vector is $g$-timelike (viz.~$g'$ has strictly narrower light cones than $g$). Given a continuous section $S$ of the vector bundle $T^0_2M$ and a compact set $C\subset M$, we define
\begin{align*}
\Vert S\Vert_{\infty,C} := \sup_{x\in C} \sup_{\substack{v\in T_xM,\\\tilde{g}(v,v)=1}} \vert S(v,v)\vert
\end{align*}
and analogously if $S$ is a continuous section of $T^0_3M$. 
Here $\tilde g$ denotes any complete smooth (background) Riemannian metric on $M$ whose topology is second countable 
\cite{NomizuOzeki61} --- considered fixed hereafter --- and
$\tilde{\nabla}$ its Levi-Civita connection. 

\begin{lemma}[Good approximation]\label{Le:Good approx} There exists a family $\{g_\varepsilon :\varepsilon > 0\}$ of smooth Lorentzian metrics, time oriented by the same vector field as $g$, such that $g_\varepsilon\prec g$ for every $\varepsilon > 0$ and $g_\varepsilon\to g$ in $\smash{C_{loc}^1}$ as $\varepsilon \to 0$. The latter means for every compact subset $C\subset M$,
\begin{align*}
\lim_{\varepsilon\to 0} \big[\Vert g_\varepsilon - g\Vert_{\infty,C} + \Vert \tilde{\nabla}g_\varepsilon - \tilde{\nabla}g\Vert_{\infty,C}\big] = 0.
\end{align*}

Moreover, if $\Rc^{(g,N,V)}\geq K$ timelike distributionally, the above family can be constructed with the following property. For every compact set $C\subset M$ and every $c,\delta,\kappa >0$, there exists $\varepsilon_0 > 0$ such that for every $\varepsilon\in (0,\varepsilon_0)$ and every $\smash{v\in TM\big\vert_C}$ with $g(v,v)\geq \kappa$ and $\tilde{g}(v,v) \leq c$, 
\begin{align*}
\Rc^{(g_\varepsilon,N,V_\varepsilon)}(v,v) \geq (K-\delta)\,g_\varepsilon(v,v).
\end{align*}
\end{lemma}

We will frequently use the approximation $g_\varepsilon$ and will refer to it as the \textit{good approximation} of $g$.

\begin{remark}[Heredity of global hyperbolicity]
\label{R:heredity of gh}
{\rm If $g$ is globally hyperbolic, and $g' \prec g$ has narrow lightcones, our definition shows $g'$ to be globally hyperbolic, as in e.g.\ Benavides-Navarro--Minguzzi \cite{NavarroMinguzzi} or Graf \cite[Rem.~1.1]{Graf20}. 
In particular, each member $g_\varepsilon$ of the good approximation is globally hyperbolic if $g$ is.} \fr
\end{remark}

\section{\texorpdfstring{Ray-adapted curves $\sigma$ and $C^1$-compactness of geodesics}{Ray-adapted curves sigma and C1-compactness of geodesics}}\label{Section: TLCRC}

Let $I\subset \R$ be an interval and $\gamma:I\to M$  a Lipschitz causal curve. By definition of time separation we have
\begin{equation}
\label{eq:elllength}\ell(\gamma_s,\gamma_t)\geq \int_s^t|\gamma'_r|\, dr\qquad\forall s,t\in I, \ s\leq t,
    \end{equation}
where here $|v|:=\sqrt{g(v,v)}$ for any $v$ future directed. We say that $\gamma$ is a {\bf maximizer} provided equality holds in \eqref{eq:elllength} for any $t,s$ as in there or, equivalently, if
\begin{equation}
\label{eq:eqmax}
\ell(\gamma_s,\gamma_t)=\ell(\gamma_s,\gamma_r)+\ell(\gamma_r,\gamma_t)\qquad\forall s,r,t\in I,\ s\leq r\leq t.
\end{equation}
A {\bf $g$-geodesic} is a $C^2$ curve $\gamma$ solving the geodesic equation 
$\nabla_{\gamma'}^g\gamma'=~0$. 
We recall 
that by \cite[Thm.\ 1.1]{LangeLytchakSaemann} any Lipschitz maximizer has a parametrization making it a $g$-geodesic (since for $C^1$-metrics Filippov solutions are classical solutions).
More generally, \cite{McCann24}  
yields a closely related criterion for a priori discontinuous curves to be locally Lipschitz and maximizing --- hence $C^2$-smooth $g$-geodesics:

\begin{lemma}[Criterion for timelike maximizers and their regularity]\label{le:regularityadapted} 
In a globally hyperbolic $C^1$-spacetime,
any --- not necessarily continuous --- map $\sigma:[0,1] \longrightarrow M$ satisfying
\begin{equation}\label{l-path}
\ell(\sigma_s,\sigma_t) = (t-s) \ell(\sigma_0,\sigma_1)>0
\qquad \forall 0 \le s< t \le 1
\end{equation}
is actually a $C^2$-smooth maximizing $g$-geodesic.
\end{lemma}

\begin{proof}
A globally hyperbolic $C^1$-spacetime is an example of a regularly localizable \cite[Cor. 2.4]{Graf20} 
and $\mathcal K$-globally hyperbolic 
\cite[Cor.\ 3.4]{Saemann16}
Lorentzian length space \cite{KunzingerSaemann18}.
Thus \cite[Lem. 2 and Cor. 6]{McCann24} assert any curve
satisfying \eqref{l-path}
becomes Lipschitz and maximizing after a 
(homeomorphic) reparameterization.
Now  \cite[Thm.\ 1.1]{LangeLytchakSaemann} ensures that there is a reparametrization making it a $g$-geodesic (hence $C^2$ by the geodesic equation \eqref{geodesic eqn}); but \eqref{l-path} also yields that $|\sigma'_t|$ is constant, so neither reparametrization was actually needed.
\end{proof}

The following simple compactness properties will be useful throughout. While (i) and (ii) are adaptations of well-known facts to our $C^1$ setting,  (iii) is more subtle: it depends on the continuity of the timelike curves described by the preceding lemma.

\begin{lemma}[Local existence and $C^1$ compactness of geodesics]\label{le:C1precomp}
Let $(M,g)$ be a globally hyperbolic $C^1$-spacetime with compact sets $K\subset TM$ and $\tilde K \subset \{\ell>0\}$. Then:
\begin{itemize}
\item[(i)] There is ${\lambda}>0$ such that for any $v\in K$ there is a $g$-geodesic defined on $[0,{\lambda}]$ having $v$ as initial velocity.
\item[(ii)] The set of $g$-geodesics on $[0,{\lambda}]$ with initial velocity in $K$ is compact in $C^1([0,{\lambda}])$. Same for the subset of those that are also future directed maximizers.
\item[(iii)]
The collection of maximizing $g$-geodesics $\sigma:[0,1]\longrightarrow M$ with endpoints $(\sigma_0,
\sigma_1) \in \tilde K$ is compact in $C^1([0,1],M)$.
\end{itemize}    
\end{lemma}

\begin{proof} 
(i) By compactness we can easily reduce to the case in which  the projection $\pi^M( K)$ of $ K$ in $M$ is contained in a single coordinate chart. Then, read in coordinates the initial velocities have uniformly bounded components, thus the existence of ${\lambda}>0$ such that the geodesic equation \eqref{geodesic eqn} is solvable in $[0,{\lambda}]$ for any initial datum in $K$ is a direct consequence of Peano's theorem \cite[\S 6.10]{BirkhoffRota89}, the compactness of $K$ and continuity of the $\Gamma^i_{jk}$'s.

\noindent (ii) The same argument also ensures that $g$-geodesics $\gamma$ with $\gamma'_0\in K$ have  second derivative uniformly bounded in $[0,{\lambda}]$. Precompactness in $C^1([0,{\lambda}])$ follows. For compactness  notice that any $C^1$-limit of solutions of \eqref{geodesic eqn} is a weak solution of the same equation; moreover it is clear, by simple ODE regularity, that  weak solutions are actually strong solutions.

Compactness of maximizers also follows taking into account that  limits of maximizers are maximizers, as is clear from the characterization in \eqref{eq:eqmax}.

{
(iii) Let $\sigma^n:[0,1]\longrightarrow M$ 
be a sequence of maximizing $g$-geodesics with endpoints $(\sigma^n_0,\sigma^n_1) \in\tilde K$.  Compactness of $\tilde K$ provides a (nonrelabelled) subsequence such that
$(x,y) := \lim_n (\sigma^n_0,\sigma^n_1)$ converges. We claim a further subsequence of the $\sigma^n$ converge in $C^1([0,1])$. By the limit curve theorem (and without reparametrizing w.r.t.\ $\tilde g$-arclength, e.g.\ as in \cite[Thm.\ 2.30]{BeranOctet24+}), we know that there is a left continuous curve $\sigma$ such that ---
along a nonrelabelled subsequence ---  $\lim_{n \to\infty} \sigma^n_t = \sigma_t$ for a.e.\ $t \in [0,1]$. Since the $\sigma^n$'s are maximizers with $\lim_n\ell(\sigma^n_0,\sigma^n_1)=\ell(x,y)>0$, they are eventually timelike maximizers.

Thus for a.e.\ $s<t$ we have
\[
\ell(\sigma_s,\sigma_t)=\lim_n \ell(\sigma^n_s,\sigma^n_t)=(t-s)\lim_n\ell(\sigma^n_0,\sigma^n_1)=(t-s)\ell(x,y).
\]
Then, by left continuity we see that $\ell(\sigma_s,\sigma_t)=(t-s)\ell(x,y)$  holds for all $s<t$. Also, by the limit curve theorem we see that $\sigma_0=\lim_n\sigma_0^n$ and there is a sequence $t_n\nearrow 1$ in $[0,1]$ such that $\sigma_1=\lim_{n\to\infty}\sigma_{t_n}$. In particular, by closedness of the causal relation $\leq$, we obtain $\sigma_1\leq y$. Moreover, we have
\[
\ell(x,y)=\ell(\sigma_0,y)\geq\ell(\sigma_0,\sigma_1)+\ell(\sigma_1,y) = \ell(x,y)+\ell(\sigma_1,y).
\]
Finally, a cutting the corner argument shows $\sigma_1=y$. Applying Lemma~\ref{le:regularityadapted} yields that $\sigma$ is a $C^2$ maximizer from $x$ to $y$. In particular $\sigma$ is continuous and we have that $\sigma_t^n\to\sigma_t$ holds for all $t\in[0,1]$.

We need to improve this pointwise convergence to $C^1$ convergence. 
We shall prove $\sigma'_0 = \lim_{n\to\infty} (\sigma^n)'_0$.  After that part (ii) asserts $C^1$ convergence of $\sigma^n$  to $\sigma$ in a neighbourhood of $t=0$, and trivial adaptation to arbitrary initial (and final, using time reversal) times $t \in [0,1]$ together with compactness of the unit interval completes the proof of (iii).
Therefore, define $v_n:=(\sigma^n)'_0$.

To derive a contradiction suppose $v_n\to 0$. 
Then part (ii) of the lemma shows that for some $\lambda>0$ the curves $(\sigma^n)$ converge in $C^1([0,\lambda])$ to a curve $\eta$ with $\eta'_0=0$; but this cannot happen because we already established that the sequence converges pointwise to $\sigma$ with $\sigma'_0\neq 0$ (because $\sigma$ is timelike).

If, on the other hand $\alpha_n:=|v_n|_{\tilde g} \to \infty$, then $\tilde v_n:=\alpha_n^{-1}v_n$ converge along a (nonrelabelled) subsequence to a non-zero vector $\tilde v$.  Part (ii) of the lemma provides $\lambda>0$ such that
the rescaled curves $\tilde\sigma^n:[0,\alpha_n]\to M$ defined by $\tilde\sigma^n_t:=\sigma^n_{t/\alpha_n}$ must --- up to further subsequences --- converge in $C^1([0,\lambda])$ to some $\eta$. The limit $\eta$ must be a causal curve with $\eta'_0=\tilde v\neq 0$, thus $\eta_1\geq\eta_0$ with $\eta_1\neq\eta_0$; we can therefore find $z\in M$ such that $x=\eta_0\in I^-(z)$ and $\eta_1\notin I^-(z).$ Now fix $t\in (0,1]$ and notice that for $n$ sufficiently large we have $\alpha_nt\geq 1$ and thus $\sigma^n_t=\tilde\sigma^n_{\alpha_nt}\geq \tilde \sigma_1$. Passing to the limit in $n$ we get $\sigma_t\geq \eta_1$ so $\sigma_t\notin I^-(p)$. Since $t\in (0,1]$ was arbitrary, we have contradicted the continuity of $\sigma_t$ at $t=0$.

The preceding (two) paragraphs imply a subsequence of the $v_n$ must converge to a (nonzero) limit $v$.  For some $\lambda>0$ part (ii) of the lemma yields $C^1([0,\lambda))$ convergence 
of $\sigma^n$ along a further subsequence,  to a limit $\eta$.  But this limit must agree with the pointwise limit $\sigma$; hence $v=\eta'_0=\sigma'_0$.  Since this limit is independent of which subsequences have been chosen along the way,  we have convergence $\sigma'_0=\lim_n(\sigma^n)'_0$ along the original (pointwise convergent) sequence, to complete the proof of (iii).
}
\end{proof}

Given a future directed curve $\gamma:[0,a)\to M$ and $a\in(0,\infty]$, one consequence of global hyperbolicity is that subsequential convergence of $\lim_{t\uparrow a}\gamma_t$ implies convergence of the full limit. If no such limit exists, we say that $\gamma$ is {\bf inextendible}. If $\gamma$ is Lipschitz with $\int_0^a|\gamma'_t|_{\tilde g}\,dt=+\infty$, then we say that it is {\bf (future) complete}. 
Note that any complete f.d. curve is necessarily inextendible.

A future {\bf   ray}  $\gamma:[0,\infty) \to M$ is a future directed curve that is inextendible and maximizing. For such a ray, we define  $I^-(\gamma):=\cup_{t\geq 0}I^-(\gamma_t)$.

`Past' notions are defined analogously.  

We turn to the important concept of co-ray. Notice that in the statement below as well as through the rest of the text, when we assert precompactness in $C^1([0,a])$ or  $C^1_{loc}([0,a))$ of a collection of curves parametrized over intervals as long or longer than $[0,a)$,  we implicitly mean precompactness of their restrictions to the indicated domains.

\begin{lemma}[Co-ray existence and 
dichotomy]\label{le:corays} Let $(M,g)$ be a globally hyperbolic $C^1$-spacetime,   $\gamma:[0,\infty)\to M$  a future complete {timelike} ray and $\eta^n:[0,a_n]\to M$, $n\in\mathbf N$,  a  timelike, maximizing $g$-geodesic with endpoint $\eta^n_{a_n}=\gamma_{t_n}$ for some $t_n\uparrow+\infty$. Assume that $\lim_{n\to\infty} \eta^n_0 = z\in I^-(\gamma)$ and that $n\mapsto  (\eta^n)'_0\in TM$ admits a non-zero limit $v$.
Then $a_n\to\infty$ and $(\eta^n)$ is precompact in  $C^1_{loc}([0,a),M)$ for some $a>0$.

Moreover, there exist ${\lambda}\in(0,\infty]$ and a $C^1_{loc}([0,{\lambda}),M)$-limit $\eta$ of some subsequence so that for no ${\lambda}'>{\lambda}$ is there a  further subsequence that converges in $C^1_{loc}([0,{\lambda}'),M)$.  Then $\eta$ is a ray --- called a co-ray of $\gamma$ ---  and exactly one of the following holds:
\begin{itemize}
    \item[i)] $\eta'_t$ is timelike for every $0 \le t <\lambda$, or
    \item[ii)] $\ell(\eta_s,\eta_t)=0$ for every $0\leq s\le t<\lambda$.
\end{itemize}
\end{lemma}

\begin{proof} We assumed $z\ll \gamma_T$ for some $T>0$, hence eventually $\eta^n_0\ll\gamma_T$ as well. By maximality we have $\ell(\eta^n_0,\gamma_{t_n})=a_n|(\eta^n)'_0|$ and by reverse triangle inequality $\ell(\eta^n_0,\gamma_{t_n})\to\infty$. Also, since $(\eta^n)'_0\to v$, we have $|(\eta^n)'_0|\to |v|<\infty$. It follows that $a_n\to\infty$.
The claimed precompactness now follows from Lemma \ref{le:C1precomp}(ii) above, which also asserts any limit curve is maximizing.

To deduce the existence of $\lambda$ and $\eta$,  we will consider $C^1_{loc}$-convergence of various subsequences on intervals of different lengths. Set $\lambda_1:=a$, so
there exists a subsequence $\{\eta^{n(k)}\}_{k \in \mathbf N}$ converging on $[0,\lambda_1)$. Starting from $i=1$,  given a subsequence converging on $[0,\lambda_i)$ we let
$c_i \geq \lambda_i$ be the supremum of those numbers $c$ such that some further subsequence converges on $[0,c)$ and set 
$\lambda_{i+1}:=\lambda_i+ \min\{1, (c_i-\lambda_i)/2\}$. Then there exists subsequence converging at least on $[0,\lambda_{i+1})$ so the process can be iterated
with $\lambda_i$ nondecreasing and 
$c_i$ nonincreasing to respective limits $\lambda,c \in (0,\infty]$.  Moreover $c>\lambda$ contradicts the construction so we must have $c=\lambda$.
A diagonal process now yields a subsequence of $\{\eta^{n(k)}\}_{k \in \mathbf N}$ which converges on $[0,\lambda)$,  with no further subsequence converging on any longer interval.

Having established the existence of $\lambda$ and $\eta$,
we argue inextendibility of $\eta$ by contradiction. Therefore
assume that $y=\lim_{t\uparrow{\lambda}}\eta_t$ exists. From the compactness of the set of solutions of the $g$-geodesic equation in a neighbourhood of $y$ we first see that ${\lambda}<\infty$ (here the assumption $v\neq 0$ matters) and then that the $(\eta^n)$'s are relatively compact in $C^1_{loc}(I,M)$ for some open interval $I$ around ${\lambda}$, contradicting the supremality of $\lambda$ in the definition of $\eta$. 

It remains to prove the dichotomy. It is clear that $\eta_t'$ is future causal for every $t$. Also, since $\eta$ is a $g$-geodesic we have that $g(\eta_t',\eta_t')$ is constant. Thus if $g(\eta_t',\eta_t')=0$ for some $t$ then $g(\eta_t',\eta_t')=0$ for every $t$ and $(ii)$ follows from maximality.
\end{proof}

We call {\bf co-ray} of $\gamma$ any curve $\eta$ built as limit as in the above statement. Notice that for every $p\in I^-(\gamma)$ a {simple compactness argument} together with global hyperbolicity shows that there exists a co-ray starting from $p$.

The study of co-rays is particularly relevant in connection with the notions of {\bf Busemann function} and of {\bf $\gamma$-adapted} curve, that we now introduce.  Ray-adapted curves are 
inspired by the {\em transport rays}~\cite{EvansGangbo99} used simultaneously and independently by three groups of authors \cite{Ambrosio03,CaffarelliFeldmanMcCann00,TrudingerWang01} 
to fix the original construction of optimal transportation maps \cite{Sudakov76}, and which have subsequently proved fruitful in 
smooth \cite{FeldmanMcCann02r,Klartag17} and nonsmooth geometry --- both Riemannian \cite{CavallettiMondino17i} and Lorentzian \cite{CavallettiMondino20+}.

Given a future complete timelike ray $\gamma:[0,\infty) \to M$, define the approximate Busemann functions $\b_t:I^-(\gamma_t)\to \R$ as
\begin{equation}\label{Busemann approximate}
\b_t(x):=\ell(\gamma_0,\gamma_t) - \ell(x,\gamma_t)
\end{equation}
and then the Busemann function $\b:I^-(\gamma)\to \R$ as 
\begin{equation}\label{Busemann limit}
\b(x):=\inf_{t\geq 0} \b_t(x)=\lim_{t\uparrow\infty}\b_t(x).
\end{equation}
An application of the reverse triangle inequality shows that $b_s\geq b_t$ for $s\leq t$, so that the limit in $t$ of these functions exist. It is then clear that $\b$ is finite-valued, upper semicontinuous and 1-steep on $I^-(\gamma)$, 
the latter meaning that
\begin{equation}\label{eq: 1steepnessbusemann}
\b(y)-\b(x)\geq \ell(x,y)\qquad\forall x,y\in 
I^-(\gamma),\quad x\leq y.
\end{equation}
Notice the Busemann function $\b$ is set in a different font from the endpoint $(a,b)$.

A curve $\sigma:(a,b)\to 
I^-(\gamma)\subset M$ will be called {\bf $\gamma$-adapted} provided
\begin{equation}
\label{eq:wellsuited}
    \b(\sigma_t)-\b(\sigma_s)=\ell(\sigma_s,\sigma_t)=t-s,\quad\forall s,t\in (a,b)\ \mbox{\rm with}
    \ s < t. 
\end{equation}
The rigidity provided by equality in \eqref{eq: 1steepnessbusemann} is the source of additional regularity we shall continue to exploit --- not only for $\sigma$ but also for $\b$. Recall Lemma \ref{le:regularityadapted} implies
a $\gamma$-adapted curve $\sigma$ is necessarily a maximizing $g$-geodesic.

\begin{remark}[{Gradient flow} characterization of $\gamma$-adapted curves]{\rm
    Although being $\gamma$-adapted is a {nonsmooth} 
    concept, it is worth noticing that if $\sigma$ takes values in an open set $U$ where $\b$ is $C^1$ with $|d\b|\equiv 1$ (as will happen later in our analysis), then $\sigma$ is $\gamma$-adapted if and only if $\sigma'_t=\nabla\b(\sigma_t)$ for every $t\in(a,b)$.  Indeed, if the latter holds we have $|\sigma'_t|=|d\b|(\sigma_t)=1$ and also $\frac{d}{dt}\b(\sigma_t)=|d\b|^2(\sigma_t)=1$. Thus
\[
t-s=\b(\sigma_t)-\b(\sigma_s)\geq\ell(\sigma_s,\sigma_t)\geq \int_s^t|\sigma'_r|\,d r = t-s,
\]
forcing the validity of \eqref{eq:wellsuited}. Vice versa, if \eqref{eq:wellsuited} holds then both  $|\sigma'_t|$ and $\frac{d}{dt}\b(\sigma_t)$ exist and are identically 1, hence
\[
1=\frac{d}{dt}\b(\sigma_t)=d\b_{\sigma_t}(\sigma'_t)\geq |d\b|(\sigma_t)|\sigma'_t|=1
\]
and the equality in the reverse Cauchy-Schwarz inequality forces $\nabla \b({\sigma_t})$ and $\sigma_t'$ to be parallel. Since they both have norm 1, they  agree. 
}\fr\end{remark}

The concept of co-rays and of $\gamma$-adapted curves are related via the following:

\begin{proposition}[Co-rays are tangent to $\gamma$-adapted curves]\label{Proposition: generalizedcoraysalongrayagreewithray}
    Let $(M,g)$ be a globally hyperbolic $C^1$-spacetime, $\gamma:~[0,\infty) \to~M$ a future complete timelike ray and $\sigma:(a,b)\to M$ a $\gamma$-adapted curve.
    Then for every $t\in(a,b)$ and co-ray $\eta$ starting from $\sigma_t$ we have $\eta'_0=k\sigma'_t$ for some $k\in(0,+\infty)$.
    \end{proposition}

\begin{proof}
  To derive a contradiction,  suppose not.  Recall
  the co-ray $\eta:[0,\lambda) \longrightarrow M$ is inextendible
  beyond some $\lambda \in (0,\infty]$.
  By a standard `cutting the corner' argument, the curve obtained by following $\sigma$ up to time $t$ and then $\eta$, cannot be a maximizer. The contradiction will come by showing instead that 
\begin{equation}
\label{eq:aligned}
\ell(\sigma_s,\eta_c)\leq \ell(\sigma_{s},\sigma_t)+\ell(\sigma_t,\eta_c)
\qquad\forall  s\in (a,t),\ c\in(0,{\lambda}).
\end{equation}
By the definition of co-ray, there are maximizing $g$-geodesics $\eta^n:[0,a_n]\to M$, $n\in\mathbf N$, converging to $\eta$ in $C^1_{loc}([0,{\lambda}),M)$ with $\eta^n_{a_n}=\gamma_{t_n}$ for some $t_n\uparrow\infty$. Fixing 
$r\in (a,b)$ with $r>t$, notice that eventually we have $\eta^n_0\in I^-(\sigma_r)$ and that
\begin{equation}
\label{eq:brigid}
\begin{split}
\lims_{n\to\infty}\ell(\sigma_t,\gamma_{t_n})-\ell(\eta^n_0,\gamma_{t_n})&\leq 
\lims_{n\to\infty}\ell(\sigma_t,\gamma_{t_n})-\ell(\eta^n_0,\sigma_r)-\ell(\sigma_r,\gamma_{t_n})\\
&\stackrel{\eqref{Busemann limit}}=\b(\sigma_r)-\b(\sigma_t)-\ell(\sigma_t,\sigma_r)
\\&\stackrel{\eqref{eq:wellsuited}}= 0.
\end{split}
\end{equation}
Also, the definition \eqref{Busemann limit} of $\b$ gives
\begin{equation}
    \label{eq:fromdefb}
    \begin{split}
\lims_{n\to\infty}    \b(\gamma_{t_n})
-\ell(\sigma_t,\gamma_{t_n})&=\lims_{n\to\infty}    \ell(\gamma_0,\gamma_{t_n})
-\ell(\sigma_t,\gamma_{t_n})
\\&=\b(\sigma_t).
\end{split}
\end{equation}
Thus we have
\[
\begin{split}
\ell(\sigma_s,\eta^n_c)   &\leq \ell(\sigma_s,\gamma_{t_n})-\ell(\eta^n_c,\gamma_{t_n})\\
     \text{(by \eqref{eq: 1steepnessbusemann})}\qquad &\leq \b(\gamma_{t_n})-\b(\sigma_s)-\ell(\eta^n_c,\gamma_{t_n})\\
    \text{(by \eqref{eq:fromdefb})}\qquad&\leq \ell(\sigma_t,\gamma_{t_n})+\b(\sigma_t)-\b(\sigma_s)-\ell(\eta^n_c,\gamma_{t_n})+o(1)\\
    \text{(by \eqref{eq:brigid} and \eqref{eq:wellsuited})}\qquad&\leq \ell(\eta^n_0,\gamma_{t_n})+\ell(\sigma_s,\sigma_t)-\ell(\eta^n_c,\gamma_{t_n})+o(1)\\
    \text{($\eta^n$ maximizing)}\qquad&=\ell(\eta^n_0,\eta^n_c)+\ell(\sigma_s,\sigma_t)+o(1).
\end{split}
\]
Passing to the limit in $n$ we get \eqref{eq:aligned} and the conclusion.
\end{proof}
This last result is the basis on which the following compactness statement is built. It will be a starting point for our study of the regularity of the Busemann function $\b$.

\begin{proposition}[From $\gamma$-adapted curves to compactness in $C^1$
] \label{Prop: boundinitialtangents}
Let $(M,g)$ be a globally hyperbolic $C^1$-spacetime, $\gamma:~[0,\infty) \to~M$ a future complete timelike ray, $\sigma:(a,b)\to M$ a $\gamma$-adapted curve \eqref{eq:wellsuited}, $\bar t\in(a,b)$
{and $\T>0$ with $\ell(\sigma_{\bar t},\gamma_\T)>0$.}
    Then there is a neighbourhood $U$ of $\sigma_{\bar t}$ such that:  for small enough ${\lambda} >0$, the collection $\mathcal G$ of maximizing $g$-geodesics parametrized by  $g$-arclength  starting from $U$ and ending in  $\{\gamma_t:t\geq{\Lambda}\}$ is precompact in $C^1([0,{\lambda}],M)$. Also, the $C^1$-closure of $\mathcal G$ only contains timelike maximizers.
\end{proposition}

\begin{proof} 
Indirectly assume that no such neighbourhood exists. So, let $x_n\to\sigma_{\bar t}$ and for every $n\in\mathbf N$ let $\eta^n:[0,a_n]\to M$ be a maximizing $g$-geodesic parametrized by $g$-arclength from $x_n$ to $\gamma_{t_n}$ for some $t_n\geq{\Lambda}$. Put $v_n:=(\eta^n)'_0$. We shall prove --- after passing to a non-relabeled subsequence --- that the $v_n$'s converge to a timelike vector.  By the arbitrariness of $(x_n)$ and $(\eta^n)$, a diagonalization argument and  Lemma \ref{le:C1precomp}(ii) this  suffices to conclude.

Possibly passing to a  subsequence we  distinguish two cases.

{\sf CASE 1}: $t_n\leq {\Lambda}'$ for every $n$ and some ${\Lambda}'<\infty$. 

Let $K$ denote the closure of $\{x_n\}_{n \in \mathbf N}$.
In this case  the $\eta^n$'s  have image contained in the emerald $J^+(K)\cap J^-(\gamma_{{\Lambda}'})$, which is compact, by global hyperbolicity \cite[Cor.\ 3.4]{Saemann16}.  Also, on the  compact set $\{(x,\gamma_t):x\in K,\ t\in[{\Lambda},{\Lambda}']\}$ the time separation is positive and finite, hence contained in some interval  $[{\lambda},\beta]\subset(0,\infty)$, say. Thus $a_n\in [{\lambda},\beta]$ for every $n$ so that after uniformly biLipschitz affine reparametrization  we can assume that  all the curves are defined in $[0,1]$. Then Lemma \ref{le:C1precomp}(iii) ensures that the $\eta^n$'s are precompact in $C^1([0,1],M)$, hence $\{v_n\}_{n\in\mathbf N}$ is precompact in $TM$ and then clearly any limit $\eta$ is  a  maximizing $g$-geodesic with  $\eta_0\in K$ and $\eta_1\in\{\gamma_t:t\in[{\Lambda},{\Lambda}']\}$. Thus $\ell(\eta_0,\eta_1)>0$ and $\eta$ is timelike, as desired.

{\sf CASE 2}: $t_n\to+\infty$. 

Since $\eta^n_0\to\sigma_{\bar t}$, possibly passing to a subsequence we can find a sequence $(\alpha_n)\subset [0,\infty]$ admitting limit $\alpha\in[0,\infty]$ such that the sequence $n\mapsto \tilde v_n:=\alpha_n v_n$ has a non-zero limit $\tilde v$.  (For example,  one may introduce any Riemannian metric $\tilde g$ on $M$ and choose $\alpha_n = |v_n|^{-1}_{\tilde g}$.)

To conclude it is  sufficient to prove that $\alpha\in(0,\infty)$.

Notice that since $|v_n|:= \sqrt{g(v_n,v_n)}=1$, no subsequence of the $v_n$'s can converge to the zero vector. Hence $\alpha\neq\infty$.

Also, since $|\tilde v|=\lim_n|\tilde v_n|=\lim_n\alpha_n=\alpha$, to prove that $\alpha\neq 0$ it suffices to show that $\tilde v$ is timelike.

With this aim, let us  introduce the auxiliary curves $\tilde\eta^n:[0,\tfrac{a_n}{\alpha_n}]\to M$ as $\tilde\eta^n_t:=\eta^n_{\alpha_nt}$, and notice that $(\tilde\eta^n)'_0=\tilde v_n\to\tilde v\neq 0$. By Lemma \ref{le:C1precomp}(ii) in some right-neighbourhood $[0,{\lambda})$ of 0 we have $C^1_{loc}$ convergence of the $\tilde\eta^n$'s to a limit $\tilde\eta$: our goal is to prove that $\tilde\eta$ is timelike.

Possibly passing to a further subsequence and picking a bigger ${\lambda}$ we can assume that for no ${\lambda}'>{\lambda}$ are  the $\tilde\eta^n$'s convergent in $C^1_{loc}([0,{\lambda}'),M)$.

In this case the limit curve $\tilde\eta$ is by definition a co-ray starting from $\sigma_{\bar t}$, hence timelike by Proposition \ref{Proposition: generalizedcoraysalongrayagreewithray}.
\end{proof}

\section{\texorpdfstring{The Busemann function is superdifferentiable near $\sigma$}{The Busemann function is superdifferentiable near sigma}}\label{Section: EquisuperdiffBusemannlimits}

In this section we start making more quantitative regularity estimates on the Busemann function, the main result being Proposition \ref{P2:semiconcave}, where its superdifferentiability  is shown in the proximity of a $\gamma$-adapted segment $\sigma$. As before, we continue to fix an auxiliary complete Riemannian metric tensor $\tilde g$ on $M$. We shall denote by $\tilde{\sf d}$ the induced distance.

Given $U\subset M$ open, a function $u:U\to \R$ is called {\bf locally Lipschitz}, provided for any $K\subset U$ compact there is $C(K)>0$ such that
\[
|u(y)-u(x)|\leq C\tilde{\sf d}(x,y)\qquad\forall x,y\in K.
\]
A family of functions on $U$ is called {\bf locally  equi-Lipschitz} if the same constant $C(K)$ can be chosen in the above inequality for all of the functions.

We say that $u$ is {\bf superdifferentiable} at $x\in U$  if there exists $ v \in T^*_xM$ such that
\begin{equation}\label{superdifferentiable}
u(\exp^{\tilde g}_x w) - u(x) - (v,w) \leq o(|w|_{\tilde g}) 
\end{equation}
as $w \to 0$. Any such $v$ is  called a superdifferential of $u$ at $x$. Here $\exp^{\tilde g}$ denotes the Riemannian exponential map and $(v,w)$ the duality pairing of cotangent with tangent vector. Analogously, a family $\{u_i\}_{i \in \mathcal I}$ is said to be 
{\bf equi-superdifferentiable} on $U \subset M$, if for any $x\in U$ and $i\in\mathcal I$ there is a superdifferential $v_{i,x}$ of $u_i$ at $x$ such that  each compact set $K \subset U$ satisfies
\begin{equation}\label{equi-superdifferentiable}
\sup_{i\in\mathcal I,\  x \in K}    \big\{u_i(\exp^{\tilde g}_x w) - u_i(x) - (v_{i,x},w) \big\}\leq o(|w|_{\tilde g}).
\end{equation} 
In other words,  if the error term in \eqref{superdifferentiable} 
is uniform over $i\in \mathcal I$, $x\in K$. Notice that in this case \emph{any} choice of superdifferentials $v_{i,x}$ does the job in the above. Similarly, a family $\{u_i\}_{i \in \mathcal I}$ is said to be {\bf equi-subdifferentiable} if $\{-u_i\}_{i \in \mathcal I}$ is equi-superdifferentiable.
(Equi-superdifferentiability of a family is a local  differential topological concept:  it depends neither on any Lorentzian metric,  nor on the choice of Riemannian metric: equi-superdifferentiability for one Riemannian metric $\tilde g$ on $M$ implies equi-superdifferentiability for every other choice of Riemannian metric as well.)

Our first task is to show the approximate Busemann functions  $(\b^+_t)_{t\ge T}$ are locally equi-Lipschitz and equi-superdifferentiable   in a neighbourhood of the image of  a $\gamma$-adapted curve. {This essentially reduces to showing these regularities for negatives of Lorentz distance functions, as stated in Proposition \ref{Pr:Equisupdist}. For a fixed Lorentzian metric $g=g_0$ this is proved} in the spirit of Villani \cite[Proposition 10.15(iii)]{Villani09} and McCann \cite[Theorem 3.6]{McCann20}. 
For later use in Section \ref{Section: pdAlembcomp} however, we merge this with the observation that our equi-superdifferentiability and local equi-Lipschitz continuity statements hold with error terms that are independent of $\varepsilon$
across the good approximation $\{g_\varepsilon:\varepsilon > 0\}$ of $g$ from Lemma \ref{Le:Good approx}. We adopt the following notational convention. Whenever an object is tagged with an $\varepsilon$, it refers to the respective quantity induced by $g_\varepsilon$; otherwise, it is understood with respect to $g$.  We also write $g_0 := g$.

\begin{proposition}[Equi-superdifferentiable Lorentz distance functions]\label{Pr:Equisupdist}   Let $(M,g)$ be a globally hyperbolic $C^1$-spacetime. Let $\{g_\varepsilon : \varepsilon > 0\}$ be a good approximation of $g$ as in Lemma~\ref{Le:Good approx}. As above, we set $g_0:= g$. Let $K,H\subset M$ be  compact so that $K\times H\subset \{\ell>0\}$.
Then there is $\varepsilon_0>0$ such that the family $\{-\ell_\varepsilon(\cdot,o) : \varepsilon\in [0,\varepsilon_0),\ o\in H\}$ is equi-Lipschitz and equi-superdifferentiable on $K$. Also,  $\vert d\ell_\varepsilon(\cdot,o)\vert _{g_\eps}= 1$ a.e.~on $K$ for every $\varepsilon\in [0,\varepsilon_0)$ and $o\in H$.
\end{proposition}

\begin{proof} 
The idea underlying this proof is that we have enough compactness and smoothness to quantify the approximation provided by the first variation formula,
even if we lack sufficient smoothness to derive a second variation formula.  We exploit the fact that a continuous derivative on a compact set has a uniform
modulus of continuity.

Recall  that  the Lagrangian $L_\varepsilon:TM\to\R$ defined as $L_\eps(v): = -g_\varepsilon(v,v)^{1/2}$ (which is of regularity $C^1$ on the timelike future bundle $T^\eps_+M\subset T^0_+M$) induces the time  separation $\ell_\eps$ via 
\begin{equation}
    \label{2global time separation}
\ell_\varepsilon(x,y) := \sup_{\sigma(0) = x, \sigma(1)=y}\int_0^1 L_\eps(\sigma'_s)\,ds.
\end{equation}
Moreover $\ell_\varepsilon\to \ell$ uniformly as $\varepsilon\to 0$ on the compact subset $K\times H$ of $\{\ell > 0\}$, cf.~Braun--Calisti \cite[Cor.~3.5]{BraunCalisti2024}. Thus, there exists $\varepsilon_0> 0$ such that $\inf_{\varepsilon\in(0,\varepsilon_0)} \inf_{x\in K,\,o\in H} \ell_\varepsilon(x,o) > 0$. 

Hence  a simple variant of Lemma \ref{le:C1precomp}(iii) --- easily achievable by inspecting the proof --- asserts the collection $\mathcal G$ of maximizing $g_\eps$-geodesics on $[0,1]$,  $\eps\in[0,\eps_0]$, starting in $K$ and ending in $H$ is $C^1([0,1],M)$ compact and also  $\hat K:=\overline{\{\sigma_s':\sigma\in\mathcal G,\ s\in[0,1]\}}\subset  T^0_+M$. In particular there is $\hat U\supset \hat K$  open,  whose closure is compact and  contained in $T_+M$.

Since $K$ is compact, assume 
it can be covered by a single smooth coordinate chart $(x^1,\ldots,x^n)$ without loss of generality.
To construct suitable variations near $x$ of the geodesic curve attaining the maximum in \eqref{2global time separation},
multiply each coordinate vector field $\frac{\p}{\p x^i}$ with a smooth cutoff function to obtain a smooth compactly supported vector field $w_{(i)}$ which agrees with 
$\frac{\p}{\p x^i}$ on $K$ but vanishes outside the coordinate chart as well as on $H$. 
On $TK \times M$,  define 
$W((\xi,x),y)= \sum_{i=1}^n \xi^i w_{(i)}(y)$ where $(\xi^1,\ldots,\xi^n)$ denote the coordinates
of the tangent vector $\xi \in T_xM$ in the given coordinate system.\ Notice $W((\xi,x),x)=\xi$, meaning
the vector field $W((\xi,x),\cdot)$ represents a perturbation which displaces $x$ in direction $\xi$. For fixed $(\xi,x)\in TK$ the corresponding unit variation $M\ni y\mapsto\exp^{\tilde g}_y(W((\xi,x),y))\in M$ depends smoothly on all of its arguments; we shall denote by $S_x^v(\xi)\in T_{\exp^{\tilde g}_y(W_x(\xi)(y))}M$ the $y$-differential of this map evaluated at  $v\in T_yM$. Notice that $S_x^v(0)=v$ and let $R_x^v:T_xM\to T_v(TM)$ be the $\xi$-differential of $S_x^v(\cdot)$ at $0\in T_xM$. 
Compactness of $K$ and $\hat K$ imply the existence of a constant $C>0$ independent of $x \in K$ and of $v \in \hat K$ such that
\begin{equation}
\label{eq:normR}
\|R_x^v(\tilde \xi)\|_{\tilde g}\leq C|\tilde \xi|_{\tilde g},\qquad\forall v\in \hat K, \tilde\xi \in T_xM
\end{equation}
where $\|\cdot\|_{\tilde g}$ is the Sasaki metric  on $T_v(TM)\sim_{\tilde g}T_yM\times T_yM$.

The $C^1$-convergence of $L_\eps$ to $L_0$ as $\eps\downarrow0$ implies $C^1(\hat U)$-compactness of $\{L_\eps:\eps\in[0,\eps_0]\}$, which in turn gives equi-superdifferentiability of both $L_\eps$ and $-L_\eps$. Thus from \eqref{eq:normR} we get the uniform estimate
\begin{equation}
\label{eq:Ldiff}
L_\eps(S_x^v(\xi))-L_\eps(v)-DL_{\eps,v}(R_x^v(\xi))=o(|\xi|_{\tilde g}),
\end{equation}
where the error term is independent of $\eps\in[0,\eps_0]$, $x\in K$ and $v\in \hat K$; (note that a simple compactness argument shows that for $|\xi|_{\tilde g}$ small enough we have $S_x^v(\xi)\in \hat U$ for every $v\in \hat K$, thus in particular $L_\eps(S_x^v(\xi))$ is well defined). For $\sigma\in \mathcal G$, $x\in K$ and $\xi\in T_xM$ we define the curve  $[0,1]\ni s\mapsto\sigma_{\xi,s}:={\exp^{\tilde g}_{\sigma_s}}(W((x,\xi),\sigma_s))\in M$ and notice that $\sigma_{\xi,s}'=S_x^{\sigma_s'}(\xi)$. Defining the linear map $\xi\mapsto Q_{\eps,\sigma}(\xi):=\int_0^1 DL_{\eps,\sigma'_s}(R_x^{\sigma'_s}(\xi))\,ds$, from \eqref{eq:Ldiff} we get
\[
\int_0^1L_\eps(\sigma_{\xi,s}')\,ds-\int_0^1L_\eps(\sigma_s')\,ds-Q_{\eps,\sigma}(\xi)=o(|\xi|_{\tilde g}),
\]
with error term independent of  $\eps\in[0,\eps_0]$ and $\sigma\in\mathcal G$.

To conclude, pick $x\in K$, $o\in H$, $\eps\in[0,\eps]$ and let $\sigma=\sigma(x,o)\in \mathcal G$ optimal in \eqref{2global time separation} for the definition of $\ell_\eps(x,o)$. Then for every $\xi\in T_xM$  the defining properties of $W((x,\xi),\cdot)$ 
give that $\sigma_{\xi,0}=\exp^{\tilde g}_x(\xi)$ and $\sigma_{\xi,1}=o$, i.e.\ $\sigma_\xi$ is a competitor for the definition of $\ell_\eps (\exp^{\tilde g}_x(\xi),o)$ and thus
\[
\begin{split}
\ell_\eps (\exp^{\tilde g}_x(\xi),o)-\ell_\eps (x,o)&\geq  \int_0^1L_\eps(\sigma_{\xi,s}')\,ds-\int_0^1L_\eps(\sigma_s')\,ds\\
&=Q_{\eps,\sigma}(\xi)+o(|\xi|_{\tilde g}),
\end{split}
\]
with error term independent of $x,o$ and $\eps\in[0,\eps_0]$. This establishes the desired equi-superdifferentiability of $-\ell_\eps(\cdot,o)$.

The equi-Lipschitz property is now a standard consequence of equi-superdifferentiability, as is easily verified in charts.

Finally, Rademacher's theorem implies differentiability of $-\ell_\varepsilon(\cdot,o)$ a.e. The fact that the differential has $g_\eps$-norm 1 follows from 1-steepness (w.r.t.\ $g_\eps$)  of $-\ell_\eps(\cdot,o)$ and the fact that along $g_\eps$-geodesics toward $o$ it increases with speed exactly 1. \end{proof}

\begin{proposition}[Equi-superdifferentiable  Busemann limits]
\label{P2:semiconcave}

Let $(M,g)$ be a globally hyperbolic $C^1$-spacetime, $\gamma:~[0,\infty) \to~M$ a future complete timelike ray and $\bar \sigma:(a,b)\to M$ a $\gamma$-adapted curve \eqref{eq:wellsuited}.
Then there is a neighbourhood $U$ of the image of $\bar \sigma$ on which 
the approximate Busemann functions $\{\b_t\}_{t\geq T}$ of $\gamma$ are all locally equi-Lipschitz and equi-superdifferentiable.  In particular:
\begin{itemize}
    \item[(i)] The  Busemann function $\b$ is 
    superdifferentiable and locally Lipschitz on $U$,
    \item[(ii)] $d\b_t\to d\b$ a.e.\ on $U$, and
    \item[(iii)] $|d\b|=1$ a.e.\ on $U$.
\end{itemize}
\end{proposition}
\begin{proof} Let $\bar t\in(a,b)$ and apply Proposition \ref{Prop: boundinitialtangents} to find  a compact neighbourhood  $K$  of $\bar \sigma_{\bar t}$ and  ${\lambda}>0$ so that the collection $ \mathcal G$  of maximizing $g$-geodesics parametrized with $g$-arclength from points in $K$ with endpoints in $\{\gamma_t,t\geq T\}$ is precompact in $C^1([0,{\lambda}],M)$ and moreover, such that each limit $\sigma$ of curves in $\mathcal G$ is timelike. We shall prove the conclusions (i),(ii),(iii) with $K$ in place of $U$: by the arbitrariness of $\bar t$ this suffices to conclude.

Notice that the set $H:=\overline{\{\sigma_{\lambda}:\sigma\in \mathcal G\}}$ is compact and assume for the moment that $K\times H\subset \{\ell>0\}$.

Let $x\in K$ and for $t\geq T$ pick $\sigma^{x,t}\in\mathcal G$ be starting from $x$ and ending in $\gamma_t$. 
Setting $u^{x,t}(\cdot):=\ell(\gamma_0,\gamma_t)-\ell(\cdot,\sigma^{x,t}_{\lambda})-\ell(\sigma^{x,t}_{\lambda},\gamma_t)$, the reverse triangle inequality yields
\begin{equation}
    \label{eq:usupport}
    \begin{split}
    \b_t(x')&\leq u^{x,t}(x')\qquad \forall x'\in K,\\
    \b_t(x)&= u^{x,t}(x)
    \end{split}
\end{equation}
and thus also
\begin{equation}
    \label{eq:btuxt}
    \b_t(x')=\inf_{x\in K}u^{x,t}(x')\qquad\forall x'\in K.
\end{equation}
Since $\sigma^{x,t}_{\lambda}\in H$, Proposition \ref{Pr:Equisupdist} ensures that the family $\{u^{x,t}\}_{x \in K,t \ge T}$ of functions on $K$ is equi-Lipschitz and equi-superdifferentiable. Hence from \eqref{eq:btuxt} we see that the approximate Busemann functions $(\b_t)_{t\ge T}$ are equi-Lipschitz on $K$ and from \eqref{eq:usupport} that they are equi-superdifferentiable.

To reduce to the case $K\times H\subset \{\ell>0\}$ we argue as follows. Pick $x\in K$ and $\sigma^{x,t}$ as above and recall that our choice of $K$ ensures 
 $H_x:=\overline{\{\sigma^{x,t}_{\lambda}:t\geq T\}}$ is contained in $I^+(x)$. By continuity of $\ell$ there is a compact neighbourhood $K_x$ of $x$ such that $K_x\times H_x\subset\{\ell>0\}$ and by compactness of $K$ a finite collection $x_1,\ldots,x_n\in K$ such that $K\subset \cup_iK_{x_i}$. Then the above arguments ensures that the $\b_t$'s have the desired uniform properties in each of the $K_{x_i}$'s and hence on their finite union.

Now the fact that $\b$ is locally Lipschitz follows, as in the above, from the fact that the infimum of a family of equi-Lipschitz functions is Lipschitz. For superdifferentiability we shall verify below the claims that: any family $(v_t)\subset T^*_xM$ for which  $v_t$ is a  superdifferential of $\b_t$ at $x$  must be precompact and that any limit of subsequences as $t\uparrow\infty$ is a superdifferential of $\b$ (so that, in particular, $\b$ is superdifferentiable on $K$). 

The precompactness claim follows noticing that if $v_t$ is a superdifferential of $\b_t$ at $x\in I^-(\gamma)$, then $|v_t|_{\tilde g^*}$ is  bounded by the Lipschitz constant of $\b_t$. Now let $t_n\uparrow\infty$ be so that the superdifferentials $v_{t_n}\in T_xM$ have a limit $v$.  Equi-superdifferentiability of the $\b_t$'s  in a neighbourhood of $x$ tells that
\[
\sup_{t\geq t_0}\big\{\b_t(\exp^{\tilde g}_x(w))- \b_t(x)-(v_t,w)\big\}\leq o(|w|_{\tilde g}).
\]
Passing to the limit in this we see that $v$ is a superdifferential of $\b$ at $x$. Thus (i) is proved.

(ii) follows along the same lines taking into account Rademacher's theorem and the fact that if a function is differentiable at a point, then the differential is the only superdifferential at that point.

(iii) is a direct consequence of (ii) and $|d \b_t|=1$ a.e., which in turn follows from the final assertion of Proposition \ref{Pr:Equisupdist}. 
\end{proof}

\begin{remark}[Differentials of time separations converge a.e.]\label{Re:Semico}{\rm An analogous  argument yields the following. Let $\{g_\varepsilon: \varepsilon > 0\}$ be a good approximation of a globally hyperbolic $C^1$-Lorentzian metric $g$ on $M$. Then for every $o\in M$ and every compact subset $X\subset I^-(o)$, we have $d\ell_\varepsilon(\cdot,o) \to d\ell(\cdot,o)$ a.e.~on $X$.}\fr
\end{remark}

\section{The Busemann function is \texorpdfstring{$p$-superharmonic near $\sigma$}{p-superharmonic near sigma}}\label{Section: pdAlembcomp}

In this section we start investigating the effects of the lower Ricci bound on the Busemann function.  The goal of this part is to establish $p$-superharmonicity of $\b$ on a neighbourhood $U$ of each $\gamma$-adapted curve $\sigma$, 
cf.~Corollary \ref{L2:harmonicity} below.   As is customary in the proof of splitting theorems across different signatures, such result is obtained as limit of estimates in place for the approximate Busemann functions $\b_t$, given here in  Proposition \ref{P2:dalembert comparison}. 
In principle, the proof of this latter result {might} follow by combining (i) the nonsmooth d'Alembert comparison theorem of Beran--Braun--Calisti--Gigli--McCann--Ohanyan--Rott--S\"amann \cite{BeranOctet24+} with (ii) the main result of Braun and Calisti \cite{BraunCalisti2024}, which implies weighted spacetimes with $C^1$-metrics and distributionally nonnegative timelike Ricci curvature satisfy both the timelike measure contraction property of Cavalletti--Mondino \cite{CavallettiMondino20+} and its variant by Braun \cite{Braun22.6+}. In favor of a simpler and more streamlined presentation, we give a self-contained proof based on the approximation of  the metric tensor as in Lemma \ref{Le:Good approx}. This is naturally based on the smooth d'Alembert comparison theorem valid outside the cut locus, as in our prequel 
\cite{BGMOS24+}. It was Case \cite{Case10} who extended Eschenburg's result~\cite{Eschenburg88} to smooth Bakry--\smash{\'Emery} spacetimes. Since approximation worsens the distributional timelike Ricci bound $K=0$ (recall Lemma \ref{Le:Good approx}), we will need to extend \cite[Lem.~5.4]{Case10} to nonzero $K$. (In the unweighted case, such a result goes back to 
Treude~\cite{Treude11} and Treude--Grant \cite{TreudeGrant13}.) The proof is standard, yet we include some details.

We shortly fix some notation. Given $\kappa\in\R$ let $\sin_\kappa$ be the unique function obeying the Jacobi equation $s'' + \kappa s = 0$ with $s(0) = 0$ and $s'(0) =1$. Explicitly,
\begin{align*}
\sin_\kappa(t) = \begin{cases} 
{\kappa}^{-1/2}\sin({\kappa}^{1/2}t) & \textnormal{if }\kappa > 0,\\
t & \textnormal{if }\kappa = 0,\\ 
{|\kappa|}^{-1/2}\sinh(|\kappa|^{1/2}t 
) & \textnormal{otherwise}.
\end{cases}
\end{align*}
Let $\pi_\kappa$ denote the first positive root of $\sin_\kappa$, i.e.~$\pi_\kappa = \pi/\sqrt{\kappa}$ if $\kappa > 0$, and conventionally $\pi_\kappa = \infty$ otherwise. We also consider  the generalized cotangent function $\cot_\kappa \colon [0,\pi_\kappa)\to \R$ 
given by
\begin{align*}
\cot_\kappa(t) := \frac{\sin_\kappa'(t)}{\sin_\kappa(t)}.
\end{align*}

In the sequel, given $0\neq p< 1$ and $V\in C^1(M)$ we define
\begin{align}\label{Eq:Ops}
\begin{split}
\Box^{(g,V)}u &:= -\textnormal{div}\,\nabla u + g^*(dV,d u),\\
\Box^{(g,V)}_p u &:= -\textnormal{div}\Big(\frac{\nabla u}{\vert du\vert^{2-p}}\Big) + g^*\Big(dV, \frac{d u}{\vert du\vert^{2-p}}\Big)
\end{split}
\end{align}
at every point where these expressions make sense pointwise for the given function $u$. Here  $g^*$ is the bilinear form induced by duality by $g$ on the cotangent space, and $\smash{\textnormal{div}\,X = 
{\vert\!\det g\vert}^{-1/2}\,\partial (X^i\,
{\vert\!\det g\vert^{1/2}})/\partial x^i}$ is the usual divergence operator induced by the $C^1$-metric $g$.  Observe the operators in \eqref{Eq:Ops} simply correspond to the divergence $\nabla^{(g,V)} \cdot X := \textnormal{div}\, X - X(dV)$ induced by the weighted volume measure
\begin{align*}
\vol^{(g,V)} := \mathrm{e}^{-V}\,\vol
\end{align*}
applied to the gradient and the $p$-gradient of $u$, respectively.

\begin{lemma}[D'Alembert comparison theorem outside cut locus]
\label{Le:Strong dalembert} 
Let $(M,g)$ be a globally hyperbolic $n$-dimensional $C^\infty$-spacetime and $V\in {C}^\infty(M)$. Let $K\in\R$, $N\in[n,\infty)$  and assume that   $\smash{\Rc^{(g,N,V)}\geq K}$  holds in all timelike directions. 

Then for every $0\neq p<1$ and $o\in M$, the inequality
\begin{align*}
\square_p^{(g,V)} (-\ell(\cdot,o)) = \square^{(g,V)} (-\ell(\cdot,o)) \leq (N-1)\cot_{K/(N-1)}(\ell(\cdot,o))
\end{align*}
holds at every point in $I^-(o)$ outside the past timelike cut locus of $o$.
\end{lemma}

\begin{proof} The case $K=0$ is due to Case \cite[Lem.~5.4]{Case10}.

Let $x\in I^-(o)$ not belong to the timelike cut locus of $o$ and let $\gamma$ denote the unique past-directed  proper-time parametrized timelike geodesic from $o$ to $x$. Set $u = (-\square_p^{(g,V)} \ell(\cdot,o))\circ\gamma$. As in \cite[Lem.~5.4]{Case10}\footnote{The parameter $m$ in  \cite{Case10} corresponds to $N-n$ in our case. Moreover, Case uses the opposite signature convention on $g_{ij}$.}, on the half-open interval $(0,l(x,o)]$ we estimate
\begin{align*}
u' \leq -\Rc^{(g,N,V)}(\dot{\gamma},\dot\gamma) - \frac{1}{N-1}\,u^2 \leq -K - \frac{1}{N-1}\,u^2.
\end{align*}
On the same interval, the assignment $v(t) := (N-1)\cot_{K/(N-1)}(t)$ solves the Riccati-type ODE 
$v' = -K - \frac{1}{N-1}\,v^2$, thus by a standard Riccati comparison argument it follows that $u\leq v$ holds in every interval $[0,a)$ on which $v$ is smooth. This suffices to conclude.

We point out that  for $K>0$ the maximal such interval is $[0,\pi_{K/(N-1)})$, a fact that can be used to prove the timelike Bonnet--Myers theorem of Cavalletti--Mondino \cite[Prop.~5.10]{CavallettiMondino20+}.
\end{proof}

Let $C^1_0(U)$ denote the continuously differentiable functions with compact support in $U\subset M$.

\begin{proposition}[Weak d'Alembert comparison for Lorentz distances]\label{P2:dalembert comparison} 
Let $(M,g)$ be a globally hyperbolic $C^1$-spacetime, $V\in C^1(M)$, $K\in \R$ and $N\in[n,\infty)$. Assume that the   Bakry--\'Emery energy condition $\smash{\Rc^{(g,N,V)}\geq 0}$ holds timelike distributionally \eqref{B-E Ric}. Put $f^o:=-\ell(\cdot,o)$ and let $0\neq p<1$.

Then for every  $0 \le \phi\in C_0^1(I^-(o))$ we have:
\begin{align*}
&\int \Big[\frac{(N-1)\phi}{f^o} 
+ g^*\Big(d\phi + \phi d V, \frac{d f^o}{\vert d f^o\vert^{2-p}}\Big)\Big] d\vol^{(g,V)} \le 0.
\end{align*}
\end{proposition}
Note: the term $|df^o|$ in the statement as well we the analogous $|df_\eps^o|$ occurring in the proof is a.e.\ equal to 1. Hence in principle we could omit them in favor of simpler looking formulas. However, doing so would hide the fact that it is the ellipticity of the $p$-d'Alembertian that makes possible the proof of the global estimate (and the whole proof plan we are pursuing here), thus we prefer to keep it.

\begin{proof} Let $\{g_\varepsilon : \varepsilon > 0\}$ be the good approximation of $g$ from Lemma \ref{Le:Good approx}; since $g_\varepsilon$ is smooth and again globally hyperbolic for every $\varepsilon > 0$, the hypotheses on the timelike cut locus from our prequel  \cite[Prop.~7]{BGMOS24+} are trivially satisfied. We  apply Lemma \ref{Le:Good approx} to the choices $C := J(\supp\phi,o)$ and $\smash{\kappa := \inf_{x\in \supp\phi} \ell(x,o)^2}>0$ to deduce that for any $\delta > 0$ there exists $\varepsilon_0 > 0$ such that for every $\varepsilon\in(0,\varepsilon_0)$ and every $\smash{v\in TM\big\vert_C}$ with $g(v,v) \geq \kappa$ and $\tilde{g}(v,v)\leq c$, we have $\smash{\Rc^{(g_\varepsilon,N,V_\varepsilon)}(v,v) \geq -\delta\,g_\varepsilon(v,v)}$. Let $\ell_\varepsilon$ denote the time separation function induced by $g_\varepsilon$, and recall $\ell_\varepsilon\to \ell$ uniformly as $\varepsilon \to 0$ on the compact subset $\supp\phi\times \{o\}$ of the chronological relation $\{\ell\geq 0\}$ as in Braun and Calisti \cite[Cor.~3.5]{BraunCalisti2024}. Employing  Lemma \ref{Le:Strong dalembert} in place of Eschenburg's d'Alembert comparison theorem in our prior proof of \cite[Prop.~7]{BGMOS24+}, for $\eps\in(0,\eps_0)$ we get
\begin{align*}
\int \Big[(N-1)\phi\cot_{\frac{-\delta}{N-1}}(f_\eps^o)
+ 
\frac{g^*_\varepsilon(d\phi + \phi d V_\varepsilon, df_\eps^o)}{\vert df_\eps^o\vert_\varepsilon^{2-p}}\Big]\,d\vol^{(g_\varepsilon,V_\varepsilon)}\leq0,
\end{align*}
where $f^0_\eps:=-\ell_\eps(\cdot,o)$. Now we let $\varepsilon \to 0$ in this expression. Since $g_\varepsilon \to g$ and
$V_\varepsilon \to V$ locally uniformly,  the $\vol$-density of $\smash{\vol^{(g_\varepsilon,V_\varepsilon)}}$ converges locally uniformly to the $\vol$-density of $\smash{\vol^{(g,V)}}$. Combined with the previously mentioned locally uniform convergence of $\ell_\varepsilon$ and Lebesgue's dominated convergence theorem,
\begin{align*}
\lim_{\varepsilon\to 0}\int \cot_{\frac{-\delta}{N-1}}(f^o_\eps)\,\phi\,d\vol^{(g_\varepsilon,V_\varepsilon)} 
= \int \cot_{\frac{-\delta}{N-1}}(f^o)\,\phi\,d\vol^{(g,V)}.
\end{align*}
Moreover --- recalling from Proposition \ref{Pr:Equisupdist} that the Lipschitz constant $\textnormal{Lip}\,\ell_\varepsilon$ of $\ell_\varepsilon$ is uniformly bounded for sufficiently small $\varepsilon$ --- Remark \ref{Re:Semico} yields $d\ell_\varepsilon(\cdot,o) \to d\ell(\cdot,o)$  a.e.\ on $\supp\phi$ as $\varepsilon \to 0$. Since $g_\varepsilon \to g$ locally uniformly and $g$ is nonsingular, we see that $g^*_\eps\to g^*$ locally uniformly as well as $\varepsilon \to 0$. Hence Lebesgue's theorem again yields
\begin{align*}
\lim_{\varepsilon\to 0}\int  \frac{g^*_\varepsilon(d\phi,df^o_\eps)}{\vert df^o_\eps\vert_\varepsilon^{2-p}}\,d\vol^{(g_\varepsilon,V_\varepsilon)} 
= \int \frac{g^*(d\phi,df^o)}{\vert df^o\vert^{2-p}}\,d\vol^{(g,V)}.
\end{align*} 
With a similar argument using $d V_\varepsilon \to d V$ locally uniformly (as $V\in C^1$ is regularized by convolution in charts), we get
\begin{align*}
&\lim_{\varepsilon\to 0}\int \frac{g^*_\varepsilon(dV_\varepsilon,d f^o_\eps)}{\vert df^o_\eps\vert_\varepsilon^{2-p}}\,d\vol^{(g_\varepsilon,V_\varepsilon)}
= \int  \frac{g^*(dV,d f^o)}{\vert d f^o\vert^{2-p}}\,d\vol^{(g,V)}.
\end{align*}
This results in the inequality
\begin{align*}
&\int \Big[(N-1)\phi \cot_{\frac{-\delta}{N-1}}(f^o) 
+  \frac{g^*(d\phi + \phi dV,d f^o)}{\vert d f^o\vert^{2-p}}\Big]\,d\vol^{(g,V)} \leq 0
\end{align*}
and letting $\delta \to 0$ gives the claim.
\end{proof}

The following is now argued as in the proof of \cite[Cor.~8]{BGMOS24+},
using Proposition \ref{P2:semiconcave} and its corollaries. 

\begin{corollary}[Weak d'Alembert comparison for Busemann limits] 
\label{L2:harmonicity}
With the same assumptions and notation of Proposition~\ref{P2:dalembert comparison}, let $\gamma:[0,\infty)\longrightarrow M$ be a future complete timelike ray with Busemann function \eqref{Busemann limit} and let $\sigma:(a,b)\to M$ be $\gamma$-adapted \eqref{eq:wellsuited}.
Then there is a neighbourhood $U$ of the image of $\sigma$ such that for any $\phi\in C^1_0(U)$ nonnegative we have
\begin{equation}
    \label{eq:bsuperp}
\int g^*\Big(d\phi+\phi dV,\frac{d\b}{|d\b|^{2-p}}\Big)\,d\vol^{(g,V)}\leq 0
\end{equation}
\end{corollary}
\begin{proof} We pick as $U$ the neighbourhood of the image of $\sigma$ given by Proposition \ref{P2:semiconcave}. Fix $0 \le \phi\in C^1_0(I^-(\gamma))$ and notice that for $t$ sufficiently big its support is contained in $I^-(\gamma_t)$. Then Proposition \ref{P2:dalembert comparison} gives
\[
\int g^*\Big(d\phi+\phi dV,\frac{d\b_t}{|d\b_t|^{2-p}}\Big)\,d\vol^{(g,V)}\leq \int \phi\frac{N-1}{\ell(\cdot,\gamma_t)}\,d\vol^{(g,V)}.
\]
Letting $t\uparrow\infty$ we see that the right hand side goes to 0, while the left one (recalling the estimates and convergences in Proposition \ref{P2:semiconcave}, that can be applied thanks to our assumptions on $\phi$) goes to the left hand side of \eqref{eq:bsuperp}.    
\end{proof}

\begin{remark}[On the timelike cutlocus] {\rm In the d'Alembert comparison result of Proposition \ref{P2:dalembert comparison}, we do not presume that the cut locus of $o$ is closed or has zero measure, as we did in the smooth case \cite[Prop.\ 7]{BGMOS24+}. In fact, to our knowledge, no notion of cut locus appears in the literature of low regularity spacetime metrics. Because we now assume global hyperbolicity (which we did not in \cite[Prop.\ 7]{BGMOS24+}), these requirements on the cut locus are trivially satisfied by the globally hyperbolic good approximations from Lemma \ref{Le:Good approx}. Since Proposition \ref{P2:dalembert comparison} is proved using these approximations, this suffices to get the desired conclusion.}\fr
\end{remark}

\section{The Busemann function is \texorpdfstring{$p$-harmonic near $\sigma$}{p-harmonic near sigma}}\label{Section: pharmonicity}

In this section we start working with a complete timelike line, rather than just with a future complete timelike ray as before.

Recall that a complete timelike line in $M$ --- hereafter simply a {\bf line} --- is a curve $\gamma:\R\to M$ such that
\[
\ell(\gamma_s,\gamma_t)=t-s\qquad\forall s,t\in\R,\ s<t.
\]
Lemma \ref{le:regularityadapted} shows such a line to be a maximizing $g$-geodesic both future and past complete.

To such $\gamma$ we associate two Busemann functions  (one forward, one backward) as follows. We first define for $t\geq 0$ the functions 
\begin{equation}\label{approximate bi-Busemann}
\begin{split}
    \b_t^+(x)&:=\ell(\gamma_0,\gamma_t)-\ell(x,\gamma_t)=t-\ell(x,\gamma_t),\\
    \b_t^-(x)&:=\ell(\gamma_{-t},x)-\ell(\gamma_0,\gamma_t)=\ell(\gamma_{-t},x)-t.
\end{split}
\end{equation}
Then we put 
\begin{equation}\label{bi-Busemann limit}
\begin{split}
\b^+(x)&:=\lim_{t\to+\infty}\b_t^+(x)=\inf_{t\geq 0}\b_t^+(x),\\
\b^-(x)&:=\lim_{t\to+\infty}\b_t^-(x)=\sup_{t\geq 0}\b_t^-(x).
\end{split}
\end{equation}
The reverse triangle inequality shows that the $\inf$ and $\sup$ in these formulas coincide with the limits, making  clear that the functions $\b^\pm$ are well defined and finite on $I(\gamma):=I^-(\gamma)\cap I^+(\gamma)=\cup_{t\geq 0}(I^-(\gamma_t)\cap I^+(\gamma_{-t}))$. Another consequence of the reverse triangle inequality is that
\[
\b^+(x)\geq \b^-(x)\qquad\forall x\in I(\gamma)
\]
and a direct computation shows that
\[
\b^+(\gamma_t)=\b^-(\gamma_t)=t\qquad\forall t\in\R.
\]
The results established in the previous sections apply directly to $\b^+$, it being the Busemann function associated to the restriction of the complete timelike line to $[0,+\infty)$. Analogous statements are in place for $\b^-$ --- with similar proofs (or by time reversal) ---  provided we deal with the `past' version of $\gamma$-adapted segment.

In fact we shall be even more restrictive and introduce the notion of segment adapted to the complete timelike line $\gamma$ (as opposed to the future/past complete ray): we say that a 
curve $\sigma:(a,b)\to M$ is {\bf $\gamma$-adapted} to the line $\gamma$ provided \eqref{eq:wellsuited} holds and moreover we have the additional rigidity
\begin{equation}
\label{eq:bpbm}
\b^+(\sigma_t)=\b^-(\sigma_t)\qquad\forall t\in(a,b).
\end{equation}
When dealing with a complete timelike line $\gamma$ instead of a ray, we shall always refer to this stricter notion, so we hope no confusion can occur. 
Notice that, trivially from the above considerations, $\gamma$ itself is $\gamma$-adapted.

It is then clear that in the proximity of a $\gamma$-adapted segment   the function $\b^-$ is locally Lipschitz, subdifferentiable and $p$-subharmonic (meaning $-b^-$ is locally Lipschitz, superdifferentiable
and $p$-superharmonic, provided we disregard the fact that
increasing and decreasing functions cannot both belong to the same 
domain of ellipticity for the $p$-d'Alembert operator $\square_p$).

\bigskip

Our next goal is to prove that $\b^+=\b^-$ in the proximity of a $\gamma$-adapted segment: this will follow from a strong tangency principle related to the ellipticity of the $p$-d'Alembertian. Note from \eqref{Eq:Ops} that the additional weight $V$ only contributes (locally bounded) lower order terms to the  weighted $p$-d'Alembertian it induces, so the argument can be given in a similar vein as in the unweighted case \cite[Prop.\ 9]{BGMOS24+}.
\begin{proposition}[Strong tangency principle]
\label{Proposition: Strongtangencyprinciple}
Under the assumptions and notation of Theorem \ref{T2:weighted splitting}, let $\sigma$ be $\gamma$-adapted.

Then  there is a neighbourhood of the image of $\sigma$   on which  $\b^+ = \b^-$.
\end{proposition}

\begin{proof}
We proceed along the lines of \cite[Prop.\ 9]{BGMOS24+}. By Proposition~\ref{P2:semiconcave}, Corollary \ref{L2:harmonicity} and their `past analogs' we know that there is a neighbourhood $W$ of the image of $\sigma$ where $\b^+,\b^-$ are both Lipschitz, with $|d \b^\pm|=1$ a.e., $\b^+$ is weakly $p$-superharmonic (i.e. \eqref{eq:bsuperp} holds) and $\b^-$ is weakly $p$-subharmonic
(meaning the opposite inequalities hold).

    Set $u:=\b^+ - \b^-$, and $\b(t):= \b^- + tu$ (so that $\b(0) = \b^-$ and $\b(1) = \b^+$) and let  $U \subset W$ be a coordinate chart that is diffeomorphic to $\Omega \subset \R^n$. Then  for all $0 \leq \phi \in C^1_0(U)$ we have
    \begin{align*}
        0 \leq& - \int_U d\vol^{(g,V)} \int_0^1 \frac{d}{dt} g^*(d\phi + \phi dV, |d\b|^{p-2} d\b) dt\\
        =& - \int_U d\vol^{(g,V)} \int_0^1 |d\b|^{p-2}g^*(d\phi + \phi dV, ((p-2) \frac{d\b \otimes \nabla \b}{|d\b|^2} + I)du) dt\\
        =& \int_\Omega dx\, e^{-V}\sqrt{|g|} \partial_i \phi \partial_j u \int_0^1 |d\b|^{p-2} \left((2-p) \frac{\partial^i \b \partial^j \b}{|d\b|^2} - g^{ij}\right) dt\\
        &+ \int_\Omega dx\, e^{-V} \phi \partial_i V \partial_j u \int_0^1 |d\b|^{p-2} \left((2-p) \frac{\partial^i \b \partial^j \b}{|d\b|^2} - g^{ij}\right)dt.
    \end{align*}
    Viewing the coefficients as frozen, $u$ is thus a supersolution $Lu \geq 0$ of a linear operator $L$ in divergence form, i.e.,
    \begin{align*}
        Lu = - \partial_i (a^{ij} \partial_j u) +  c^j \partial_j u,
    \end{align*}
    with $L^\infty_{loc}$-coefficients
    \begin{align*}
        a^{ij}(x) &:= e^{-V} \sqrt{|g|} \int_0^1 |d\b|^{p-2} \left((2-p) \frac{\partial^i \b \partial^j \b}{|d\b|^2} - g^{ij}\right)dt,\\
        c^j(x) &:= 
        e^{-V} \sqrt{|g|} \partial_i V \int_0^1 |d\b|^{p-2} \left((2-p) \frac{\partial^i \b \partial^j \b}{|d\b|^2} - g^{ij}\right)dt=a^{ij} \p_i V.    \end{align*}
    Note that the expression in the brackets of $a^{ij}$ is the matrix $\mathrm{diag}(1-p,1,\dots,1)$ at $\sigma(t_0)$ in any orthonormal coordinate system with $\partial_1|_{\sigma(t_0)} = \sigma'(t_0)$, which is positive definite precisely because $p < 1$. This, together with local Lipschitz continuity of $\b^\pm$  can now be used to derive ellipticity of $L$ (after shrinking $W$ if necessary), see the discussion in \cite[Rem.\ 10]{BGMOS24+}. At this point, since $u \geq 0$, $Lu=0$, and $u\circ \sigma = 0$, the strong tangency principle (see \cite[Thm.\ 8.19]{GilbargTrudinger98}) yields $u=0$, i.e., $\b^+ = \b^-$ on $W$. 
\end{proof}

To deduce higher regularity of $\b^\pm$, we need the following generalization of Evans \cite[Sec.\ 8.3.1, Thm.\ 1]{EvansPDE2ndEd}, whose proof is straightforward but included for convenience. In this lemma and its proof, we denote coordinate derivatives by subscripts having a further super- or subscript of their own.

\begin{lemma}[Using ellipticity to gain a second Sobolev derivative]\label{le:evans}
    Let $U, \tilde{U} \subset \R^n$ be open and bounded, with $\tilde{U}$ convex. Let $H:U \times \tilde{U} \to \R$ be $C^1$, as well as smooth in the second variable.  
     Denote $(x,v) = (x^1,\dots,x^n,v_1,\dots,v_n) \in U \times \tilde U$ and assume there exist $C \ge \theta >0$ such that all $(x,v,\xi)\in U \times \tilde{U} \times \R^n$  
     satisfy
    \begin{align*}
        |\xi|^2 \theta \leq \sum_{i,j} H_{v_i,v_j}(x,v) \xi_i \xi_j \leq C |\xi|^2.
    \end{align*}
    Suppose that $u \in C^1(U)$ is a weak solution of
    \begin{align*}
        \sum_{i=1}^n (H_{v_i}(x,du(x)))_{x^i} = 0,
    \end{align*}
    meaning all $\phi \in W^{1,2}_0(U)$ satisfy
\begin{equation}
\label{eq:weakform}
        \int_U \sum_{i=1}^n H_{v_i}(x, du(x)) \phi_{x^i} dx = 0.
\end{equation}
    Then $u \in W^{2,2}_{loc}(U)$.
\end{lemma}
\begin{proof}
    We follow the proof of \cite[Sec.\ 8.3.1, Thm.\ 1]{EvansPDE2ndEd}. Given $W \subset U$ open, we suppose it is precompact in $\tilde{W}$ which is open and precompact in $U$. We fix a smooth cutoff $\zeta$ which is $1$ on $W$, vanishes outside of $\tilde{W}$, and satisfies $0 \leq \zeta \leq 1$. Given $|h| > 0$ small, we choose as test function $\phi = -\Delta^{-h}_k(\zeta^2 \Delta^h_k u)$, where $\Delta^h_k u(x):=\tfrac{u(x+he_{(k)})-u(x)}h$ is the difference quotient of $u$ of parameter size $h$ in the $k$-th coordinate direction $e_{(k)}$. Inserting $\phi$ in \eqref{eq:weakform} and using `integration by parts' on the difference quotient yields
    \begin{align}\label{eq: differencequotform}
        \sum_{i=1}^n \int_U \Delta^h_k(H_{v_i}(x,du(x))(\zeta^2 \Delta^h_k u)_{x^i} = 0.
    \end{align}
    We calculate now as follows:
    \begin{align*}
        &\Delta^h_k H_{v_i}(x,du(x)) 
        \\&= \frac{H_{v_i}(x + he_{(k)}, du(x+he_{(k)}) - H_{v_i}(x, du(x))}{h} \\
        &= \frac{1}{h} \int_0^1 \frac{d}{ds} H_{v_i}(x+ s he_{(k)}, 
        (1-s) du(x)+ s(x + he_{(k)}) ) ds\\
        &= \frac{1}{h} \int_0^1 \sum_{j=1}^n H_{v_i,v_j} (u_{x^j}(x+he_{(k)}) - u_{x^j}(x))ds 
        +\int_0^1 \partial_{x^k} H_{v_i} ds\\
        &=: \sum_{j=1}^n a^{ij,h} \Delta^h_k u_{x^j}(x) + c^{i,h},
    \end{align*}
    where
    \begin{align*}
        a^{ij,h} := \int_0^1 H_{v_i,v_j}(x+s he_{(k)}, (1-s) du(x)+sdu(x+he_{(k)})  ) ds,\\
        c^{i,h} := \int_0^1 \partial_{x^k} H_{v_i}(x+she_{(k)}, (1-s) du(x)+sdu(x+he_{(k)}) ) ds.
    \end{align*}
    Substituting this into \eqref{eq: differencequotform} and using the product rule yields
    \begin{align*}
        0 = A_1 + A_2 + A_3 + A_4,
    \end{align*}
    where
    \begin{align*}
        A_1 &= \sum_{i,j} \int_U \zeta^2 a^{ij,h} (\Delta^h_k u_{x^j}) (\Delta^h_k u_{x^i}) dx,\\
        A_2 &= \sum_{i,j} \int_U a^{ij,h} (\Delta^h_k u_{x^j}) (\Delta^h_k u) 2 \zeta \zeta_{x^i} dx,\\
        A_3 &= \sum_i \int_U \zeta^2 c^{i,h} (\Delta^h_k u_{x^i}) dx,\\
        A_4 &= \sum_i \int_U c^{i,h} (\Delta^h_k u) 2\zeta \zeta_{x^i} dx.
    \end{align*}
    The terms $A_1$ and $A_2$ are estimated just as in \cite[Sec.\ 8.3.1, Thm.\ 1]{EvansPDE2ndEd}, using the uniform convexity assumptions on $H$. 
    Since $c^{i,h}$ is continuous and $u$ is $C^1$, we can estimate
    \begin{align*}
        |A_3| &\leq \sum_i \int_U \zeta^2 |c^{i,h}| |\Delta^h_k du| dx\\
        &\leq \overline{C} \int_U \zeta^2 |\Delta^h_k du| dx \\
        &\leq \overline{C} \bar{\varepsilon} \int_U \zeta^2 |\Delta_k^h du|^2 dx + \frac{\overline{C}}{\overline{\varepsilon}} \int_U \zeta^2 dx
        \\ \mbox{\rm 
        and} 
        \quad |A_4| & \le \overline{C} 
        \int_U \zeta^2 |\Delta_k^h u|^2 dx + 
        {\overline{C}}
        \int_U (\zeta_{x^i})^2 dx
            \end{align*}
similarly.    
If we choose $\bar{\varepsilon}$ sufficiently small (depending on $\overline{C}$ and $\theta$), then these estimates, along with the ones for $A_1$ and $A_2$, give an estimate on $\int_W |\Delta^h_k du|^2 dx$ in terms of $\int_W |du|^2 dx$ plus a constant. Since this estimate is independent of $h$, by letting $h\downarrow0$ we conclude that  $u \in W^{2,2}(W)$, as desired.
\end{proof}
The next corollary collects the regularity information we know so far on $\b^\pm$.

\begin{corollary}[$p$-harmonicity of the Busemann function]
\label{Corollary: W22loc}
Under the assumptions and notation of Theorem \ref{T2:weighted splitting}, let $\sigma$ be $\gamma$-adapted.
Then there is an open neighbourhood $W$ of the image of $\sigma$ such that $\b^+=\b^-$ in $W$, these functions are in  $C^1(W) \cap W^{2,2}_{loc}(W)$ and $p$-harmonic, i.e.\ equality holds in \eqref{eq:bsuperp} for all $\phi\in C^1_0(W)$.
\end{corollary}
\begin{proof}
    We already know that on some neighbourhood $W$ of the image of $\sigma$ we have $\b^+=\b^-$. Then $p$-harmonicity follows from super/sub $p$-harmonicity of $\b^+$ and $\b^-$. Super/sub differentiability of these functions together with closure of super/sub differentials yield $C^1$ regularity. Finally, for $W^{2,2}_{loc}$ regularity we apply Lemma \ref{le:evans} above with $H(x,v) =- \frac1p (g^{ij}(x) v_i v_j)^{p/2}$ and $u = \b^\pm$.
\end{proof}

\section{The Busemann function has vanishing Hessian near \texorpdfstring{$\sigma$}{sigma}}
\label{Section: HigherregBochnerOhta}

Recall that on a smooth spacetime $(M,g)$, 
whose cotangent bundle is equipped with the Hamiltonian $H(v)=f(|v|)$ with $f\in C^\infty((0,\infty))$ for $v$ future timelike, we have the following version \cite[Lemma 12]{BGMOS24+} of the Bochner--Ohta identity \cite{Ohta14h} for functions $u \in C^3(\bar{M})$ with $|du| \geq \delta > 0$:
\begin{align}
    &\nabla \cdot [D^2 H|_{du} d(H|_{du})] - (DH)d[\nabla \cdot (DH|_{du})] \nonumber \\
    &= Tr[(D^2 H) (\nabla^2 u) (D^2 H) (\nabla^2 u)] + \Rc(DH,DH).
\end{align}
Although $H(v,x)$ is defined on the cotangent bundle $T^*M$, the operator $D$ denotes differentiation with respect to $v$ only, whereas $d$ (and $\nabla$) denote exterior (and covariant) differentiation 
with respect to $x \in M$.
In the specific case $H(v) = -|v|^p/p$ ($0 \neq p < 1$) of interest to us, this can be rewritten as
\begin{align}\nonumber
& g(|du|^{p-2} du , d (\square_p u)) 
-\nabla \cdot ((D^2 H) d (|du|^p/p)) 
\\ \label{Bochner3} 
&=\mathrm{Tr}\left[ \sqrt{D^2H} (\nabla^2 u) D^2 H (\nabla^2 u) \sqrt{D^2H}\right] + |du|^{2p-4} \Rc(du, du).
\end{align}
The incorporation of an additional weight $V\in C^\infty(M)$ in this formula is straightforward. Setting $\nabla^{(g,V)} \cdot Y := \nabla \cdot Y - Y(dV)$ and adding the expression
\begin{align*}
 X^i \nabla_i(X^j\nabla_j V) - (\nabla_j V)X^i\nabla_i X^j = X^i X^j \nabla_i \nabla_j V 
\end{align*}
for the Hessian of $V$ 
applied to the vector field $X:= DH$ yields
\begin{align*}
&g(\vert du\vert^{p-2}\,du, d(\square_p^{(g,V)} u)) - \nabla^{(g,V)}\cdot ((D^2H)d(\vert du\vert^p/p))\\
&= \mathrm{Tr}[(D^2 H) (\nabla^2 u) (D^2 H) (\nabla^2 u)] + \Rc(DH,DH) +  \Hess V(DH,DH).
\end{align*}
We note that the weighted Hilbert--Schmidt norm of the Hessian of $u$ is unaffected by the weight. When integrated against a test function $\phi \in C^1_0(U)$, we get
\begin{align}
\nonumber
\int_U \bigg(\mathrm{Tr}\left[ 
\left(\sqrt{D^2H} (\nabla^2 u) 
\sqrt{D^2H}\right)^2\right] 
+ \frac{(\Rc + \nabla^2 V)(du,du)}{|du|^{4-2p}}\bigg) \phi \,e^{-V} d\vol
\\ 
\label{weighted B-O} 
=\int_U [ (\square_p^{(g,V)} u)\nabla^{(g,V)} \cdot (\phi DH )
- D^2 H (d\phi,d(H)) ] e^{-V} \, d\vol. 
\end{align}
We can use this formula and the previous information gathered about Busemann functions to  obtain the following crucial local structural result:

\begin{proposition}[Local structure]\label{Proposition: higherregularity, weightconstant}
Adopt the assumptions and notation of Theorem \ref{T2:weighted splitting}.
If $\sigma$ is $\gamma$-adapted then on a neighbourhood of its image: $\b^\pm\in C^2$ and has zero Hessian, $|d\b^\pm|\equiv 1$ and $dV(\nabla\b^\pm)=0$.
\end{proposition}

\begin{proof}
We write $u:=\b^+$ in the following. Fix any relatively compact open coordinate neighbourhood $\tilde{U} \subset W$ of the neighbourhood $W$ from Corollary \ref{Corollary: W22loc}
and $\phi \in C^1_0(\tilde{U})$. Approximate $g$ by the good approximation $g_\varepsilon$ of Lemma \ref{Le:Good approx}, $u$ by smooth functions $u_\varepsilon$ in $(C^1 \cap W^{2,2})(\tilde{U})$ such that $|du_\varepsilon|_\varepsilon \geq 1 - \varepsilon$ (since $|du| = 1$) and $V$ by the smooth weights $V_\varepsilon:=V \star \rho_\varepsilon$. In particular, we also have that $\nabla_\varepsilon^2 u_\varepsilon \to \nabla^2 u$ in $L^2(\tilde{U})$ and, as a consequence, $\square_{p}^{(g_\varepsilon,V_\varepsilon)} u_\varepsilon \to \square_p^{(g,V)} u$ in $L^2(\tilde{U})$.

Note that the Bakry--\'Emery energy condition $\Rc^{(g,N,V)}(X,X) \geq 0$ for $C^1$ timelike vector fields $X$ yields the corresponding condition for $N= \infty$, and also yields an approximate energy condition for $g_\varepsilon$ by Lemma \ref{Le:Good approx}. Since we have fixed the coordinate neighbourhood 
$\tilde{U}$ and the (Euclidean) norms $|du_\varepsilon|_{e}$ 
are all uniformly bounded on $\tilde{U}$, for any $\delta > 0$ there exists $\varepsilon_0 > 0$ such that  $\Rc^{(g_\varepsilon, \infty, V_\varepsilon)}(du_\varepsilon,du_\varepsilon) \geq - \delta $ for all $0 < \varepsilon < \varepsilon_0$. Rewriting $\vol = \sqrt{|g|} \mathfrak{L}^n$, where $\mathfrak{L}^n$ is the Lebesgue measure on $\tilde{U}$ (here identified with its image in $\R^n$), as well as using that $g_\varepsilon \to g $ in $C^1(\tilde{U})$, we see that ultimately we may pass to the limit in the following weighted Bochner--Ohta identity for the approximate metrics:
\begin{align*}
&\int_{\tilde{U}} [\phi \mathrm{Tr}\left[ 
\left(\sqrt{D^2H_\varepsilon} (\nabla^2_\varepsilon u_\varepsilon) 
\sqrt{D^2H_\varepsilon}\right)^2\right] 
+ \phi |du_\varepsilon|^{2p-4} \Rc^{(g_\varepsilon,\infty,V_\varepsilon)}] 
d\vol^{(g_\varepsilon,V_\varepsilon)}
\\&=\int_{\tilde{U}} [ (\square_{p}^{(g_\varepsilon, V_\varepsilon)} u_\varepsilon)\nabla^{(g_{\varepsilon}, V_{\varepsilon})} \cdot (\phi DH_\varepsilon )
- D^2 H_\varepsilon (d\phi,d(H_\varepsilon(d u_\eps))) ] e^{-V_\varepsilon}
d\vol_{\varepsilon},
\end{align*}
where the Ricci curvature is evaluated in $(du_\eps,d u_\eps)$. Here, $H_\varepsilon(v):=-|v|_\varepsilon^p/p$ and its
derivatives $DH_\varepsilon$ and $D^2H_\varepsilon$ with respect to $v$ are understood to be evaluated at $v=du_\varepsilon$.
Due to the convergences discussed above, passing to the limit here and noting the fact that $u$ is weighted $p$-harmonic and $|du| = 1$ we conclude 
\begin{align}\label{pre Bochner}
&\int_{\tilde{U}} \phi \mathrm{Tr}\left[ 
\left(\sqrt{D^2H} (\nabla^2 u) 
\sqrt{D^2H}\right)^2\right] \le 0
\end{align}
for $\phi \ge 0$.  Now positive-definiteness of $D^2 H$ and arbitrariness of $\phi \ge 0$ imply 
$\Hess u = 0$ a.e.\ on $\tilde{U}$. Since the metric $g$ is $C^1$, we get from the coordinate representation of the Hessian (recall \eqref{Eq:HESS}) that $u \in C^2(\tilde{U})$, hence $C^2(U)$ by arbitrariness of $\tilde{U}$. As a consequence, $\Hess u = 0$ everywhere.

Moreover,  this improved regularity $du \in C^1$ allows us to upgrade the inequality \eqref{pre Bochner} obtained in the limit $\varepsilon \to 0$ above
to a nonsmooth
weighted Bochner--Ohta identity \eqref{weighted B-O} in which the Bakry--\'Emery Ricci tensor $\Rc^{(g,\infty,V)}$ is understood in the distributional sense of Geroch and Traschen~\cite{GerochTraschen87} and all the other integrals vanish.
Arbitrariness of $\phi$ then yields
\begin{align*}
    0 = (\Rc+ \Hess V)(du,du) \geq \frac{1}{N-n} g^*(dV,du)^2 \geq 0,
\end{align*}
where the second inequality holds since $(M,g)$ satisfies the Bakry--\'Emery energy condition of Definition \ref{D:BEC}.
Hence $dV(\nabla u) = 0$, as desired.
\end{proof}

\section{The spacetime splits}\label{Section: Proof}
In this section we globalize our local structural result from Proposition~\ref{Proposition: higherregularity, weightconstant}  to obtain the full splitting theorem.

\begin{theorem}[Product structure]\label{T:productstructure}
    Adopt the assumptions and notation from Theorem \ref{T2:weighted splitting}.
Then there is a $C^2$ complete Riemannian manifold $(S,h)$ of dimension $n-1$ and a $C^1$ isometric {embedding} ${\sf Fl}:\R\times S\to M$ such that all $y \in S$ satisfy:
\begin{itemize}
    \item[-] the (open) image of $\R\times S$ via ${\sf Fl}$ is contained in $I(\gamma)$ and on it
     $\b$ is $C^2$ with vanishing Hessian and $|d\b|\equiv 1$;
     \item[-] ${\sf Fl}_t(s,y)={\sf Fl}_{t+s}(0,y)$ for  
     all $t,s\in\R$;
    \item[-] $\R\ni t\mapsto{\sf Fl}_t(y)$ is $\gamma$-adapted with  ${\sf Fl}_0(y)=y$; 
    \item[-] ${\sf Fl}_t(z)=\gamma_t$ for every $t\in\R$ and some $z\in S$;
    \item[-] $(S,h)$ has Bakry-\'Emery tensor $\Rc^{(h,N-1,V)} \ge 0$ distributionally;
    \item[-] 
    the function $V({\sf Fl}_t(y))$ is independent of $t$.
\end{itemize}
Here $\R\times S$ is equipped with the Lorentzian tensor $\hat h:=dt^2-h$ and {`isometric embedding'} means ${\sf Fl}^*g=\hat h$.
 \end{theorem}
 
\begin{proof}
{\sc Set up:}
Let $U\subset  I(\gamma)\subset M$ be the maximal connected open set containing $\gamma_0$ such that $\b^+=\b^-\in C^2(U)$ with zero Hessian and $\d V(\nabla\b^\pm)=0$: the fact that $\gamma$ is $\gamma$-adapted and Proposition \ref{Proposition: higherregularity, weightconstant} yield that one such open set exists; $U$ is therefore well defined,
since the union of all such open sets has the same properties. 
Set $\b:=\b^+=\b^-$ on $U$. Since $\nabla \b(\gamma_0)=\gamma_0'$, throughout $U$ we see that $\nabla \b$ is timelike  on $U$ with $|d \b|\equiv 1$.

The implicit function theorem ensures that $S:=U\cap\{\b=0\}$ is a $C^2$ hypersurface; equip it with the $C^1$ Riemannian metric tensor $h:=-g|_S$, call ${\sf d}$ the distance it induces, and set
\[
\bar B(r):=\{x\in S\ :\ {\sf d}(x,\gamma_0)\leq r\}
\]

For $x\in U$ let ${\Lambda}(x)>0$ be the supremum of the $t$'s such that there is a (necessarily unique) solution $(-t,t)\ni s\mapsto{\sf Fl}_s(x)\in U$ of the ordinary differential equation $\frac{d}{ds}{\sf Fl}_s(x)=\nabla \b({\sf Fl}_s(x))$ with initial condition ${\sf Fl}_0(x)=x$. 
Local solvability of this Cauchy problem implies lower semicontinuity of ${\Lambda}(\cdot)$.

Equip $\R\times S$ with the Lorentzian metric tensor $\hat h:= dt^2-h$ as in the statement and let $W\subset\R\times S$ be the set of $(t,x)$ such that $|t|<{\Lambda}(x)$. It is clear that $W$ is open and, since $\b$ has zero Hessian its (unit) gradient is a Killing vector field on $U$, so ${\sf Fl}:W\to U\subset I(\gamma)\subset M$ is an isometric embedding. The construction also ensures that all the claims  in the statement will be proved once we show that $W=\R\times S$ and that $S$ is complete.  The fact that $d V(\nabla\b)=0$ follows from the analogous claim in Proposition \ref{Proposition: higherregularity, weightconstant}.  It implies the Bakry-\'Emery lower bound $\Rc^{(h,N-1,V)} \ge 0$ on $S$ distributionally from the timelike distributional
lower bound $\Rc^{(g,N,V)} \ge 0$ on $M$,  the tensorization of Ricci curvature \cite[Corollary 7.43]{O'Neill83}, and the calculation $N-n = (N-1)-(n-1)$.

Let ${\RR}\subset [0,\infty)$ be the collection of radii $R$ for which the closed $\sf d$ ball $\bar B(R)\subset S$ centered at $\gamma_0$ is compact and $\Lambda(x)=+\infty$ for every $x\in \bar B(R)\subset S$. If we prove that ${\RR}=[0,\infty)$ we are done.  

Since $\gamma$ is $\gamma$-adapted, by Proposition \ref{Proposition: higherregularity, weightconstant} we know that  $0\in{\RR}$. We shall now show that   ${\RR}$  is both open and closed; this will suffice to conclude.

{\sc Claim 1: ${\RR}$ is open (via `vertical extension').} Let $R\in {\RR}$. By  the local compactness of $S$ and the compactness of $\bar B(R)\subset S$ we see that there is $R'>R$ such that $\bar B(R')\subset S\subset U$ is still compact. For such $R'$, by lower semicontinuity and positivity we have $\eps:=\inf_{\bar B(R')}{\Lambda}(\cdot)>0$, hence possibly picking a smaller $R'>R$ (a choice that increases $\eps$) we can assume that  $R'-R<\eps$. 

For $T\geq 0$ consider the compact subset $K_T$ of $\R\times S$ defined as
\[
K_T:=\{(t,x)\in\R\times S: x\in \bar B(R'),\ |t|+{\sf d}(x,\gamma_0)\leq T+R'\}
\]
and notice that the choice of $R'$ ensures that $K_0\subset W$. To prove our claim it suffices to check that $K_T\subset W$ for any $T\geq 0$.

Let  $\leq_W$ be the causal relation in $(W,\hat h)$ (this might be different from the restriction to $W$ of the causal relation in $\R\times S$ as we might have fewer causal curves), and  notice that if $K_T\subset W$ for some $T\geq 0$  then 
\[
(-T-R',\gamma_0)\leq_W(t,x)\leq_W(T+R',\gamma_0)\qquad \forall (t,x)\in K_T.
\]
Also, since  ${\sf Fl}:W\to M$ is an isometric embedding we have  that $(t,x)\leq_W(s,y)$ implies ${\sf Fl}_t(x)\leq{\sf Fl}_s(y)$. Thus $K_T\subset W$ implies  ${\sf Fl}(K_T)\subset D_T$,  the diamond $D_T:=J^+(\gamma_{-T-R'})\cap J^-(\gamma_{T+R'})$ being compact by global hyperbolicity.

The set $\mathcal T$ of times $T$ such that $K_T\subset W$ is (i) non-empty (since $0 \in\mathcal T$) and (ii) open (because $W$ is open and $K_T$ is compact); 
to conclude $\mathcal T =[0,\infty)$ it suffices to show that $\mathcal T$ is (iii) closed. We claim that if $K_{T_n}\subset W$ for $T_n\uparrow T$ then in fact $K_T\subset W$.

For $x\in\bar B(R')$ let $\hat T(x):=T+R'-{\sf d}(x,\gamma_0)$ and notice that the causal curve $(-\hat T(x),\hat T(x))\ni t\mapsto {\sf Fl}_t(x)$ is well defined, $\gamma$-adapted (obvious from the definition, i.e. \eqref{eq:wellsuited}, \eqref{eq:bpbm}) and  takes values in $D_{T}$. Hence it admits limits $x^+$ and $x^-$ as $t\uparrow \hat T(x)$ and $t\downarrow-\hat T(x)$, respectively. By the local structural result Proposition \ref{Proposition: higherregularity, weightconstant} and  the arbitrariness of $x\in \bar B(R')$, to conclude it suffices to show that $x^\pm$ are intermediate points of $\gamma$-adapted segments, as then it is clear that $(\pm\hat T(x),x)\in W$.

We show this for $x^+\in U$ by proving that there is $\gamma$-adapted curve $\eta:[0,{\lambda})\to M$ so that $x^+=\eta_t$ for some $t\in(0,{\lambda})$: by the local structural result Proposition \ref{Proposition: higherregularity, weightconstant} and the maximality of $U$, this will suffices to conclude. The argument for $x^-$ is analogous. Let  $\eta^n:[0,a_n]\to M$ be maximizing $g$-geodesics from $x$ to $\gamma_{t_n}$ for some $t_n\uparrow\infty$. Possibly up to subsequences and affine reparametrizations we can assume that $(\eta^n)'_0$ have a non-zero limit. Hence up to a further subsequence  we know that $(\eta^n)$ converges in $C^1_{loc}([0,{\lambda}),M)$ to some limit co-ray $\eta$ which is inextendible at $\lambda \in (0,\infty]$ by Lemma \ref{le:corays}. 
Since $(-\hat T(x),\hat T(x))\ni t\mapsto {\sf Fl}_t(x)$ is $\gamma$-adapted, by Proposition \ref{Proposition: generalizedcoraysalongrayagreewithray} we know that $\eta$ is timelike, hence up to a further affine reparametrization ({of the converging sequence and limit by proper time,} to be able to pass to the limit in the expressions below) we can assume  that $\ell(\eta_s,\eta_t)=t-s$ for any $s,t\in[0,{\lambda})$, $s<t$. 

 We now claim that $\eta$ is $\gamma$-adapted. Given what was just proved, the 1-steepness of $\b^\pm$ and the global inequality $\b^+\geq\b^-$, to check this it suffices to prove that $\b^+(\eta_s)-\b^+(x)\leq\ell(x,\eta_s)$ for any $s\in[0,{\lambda})$. Fix such $s$, pick $t\in(s,{\lambda})$ and notice that $\eta_s\ll\eta_t$ as $\eta$ is timelike. Hence for  $n$ sufficiently large we have $\eta_s\ll\eta^n_t$, thus $\ell(\eta_s,\gamma_{t_n})\geq \ell(\eta_s,\eta^n_t)+\ell(\eta^n_t,\gamma_{t_n})$ and therefore
 \[
\b^+_{t_n}(\eta_s)-\b^+_{t_n}(x)=\ell(x,\gamma_{t_n})-\ell(\eta_s,\gamma_{t_n}) \leq \ell(x,\eta^n_t)-\ell(\eta_s,\eta^n_t).
\]
Letting $n\to\infty$ we obtain $\b^+(\eta_s)-\b^+(x)\leq\ell(x,\eta_t)-\ell(\eta_s,\eta_t)=\ell(x,\eta_s)$, as desired.

We thus see that both $[0,\hat T(x))\ni t\mapsto{\sf Fl}_t(x)$ and $[0,{\lambda})\ni t\mapsto \eta_t$ are solutions of the Cauchy problem $\xi'_t=\nabla\b(\xi_t)$ with $\xi_0=0$ and since $\b\in C^2(U)$ this problem admits a unique maximal solution. It follows that $\eta_t={\sf Fl}_t(x)$ for any $t\in[0,\min\{\hat T(x),{\lambda}\})$ and since $\eta$ is inextendible at $\lambda$ 
while $[0,\hat T(x))\ni t\mapsto{\sf Fl}_t(x)$  is not (as $x^+$ exists), we see that  ${\lambda}>\hat T(x)$ and $\eta_{\hat T(x)}=x^+$, concluding Claim 1.

{\sc Claim 2: ${\RR}$ is closed (via `lateral extension').} 

Let $(R_n)\subset{\RR}$ satisfy $R_n\uparrow R$. Our goal is to prove that $R\in{\RR}$. 

By definition of geodesic distance in $S$ it is clear that any point in $\bar B(R)$ is limit of points in $B(R):=\cup_N\bar B(R_n)$, hence to prove that $\bar B(R)$ is compact it suffices to prove that any ${\sf d}$-Cauchy sequence $(x_n)$ in $B(R)$ has a limit in $\bar B(R)$.

The assumption $(x_n)\subset B(R)\subset U$ together with the previous step yield the existence of $\gamma$-adapted lines $\sigma^n:\R\to M$ with $\sigma^n_0=x_n$ (via the explicit formula $\sigma^n_t:={\sf Fl}_t(x_n)$). 

Let $T>0$ and notice that   the same arguments used in the previous step ensure that $\{\sigma^n_t:n\in\mathbf N,\ |t|\leq T\}$ is contained in the (compact) diamond $ J^-(\gamma_{R+T})\cap J^+(\gamma_{-R-T})$.  From $\ell(\sigma^n_{-T},\sigma^n_T)=2T$
it follows that the closure $\tilde K$ of
$\{(\sigma^n_{- T},\sigma^n_T)\}_{n\in\mathbf N}$ is a compact 
subset of $\{\ell>0\}$. 
Thus an application of Lemma \ref{le:C1precomp}(iii)
and a diagonalization argument show that --- after passing to a non-relabeled subsequence --- the lines $\sigma^n$ converge in $C^1_{loc}(\R,M)$ to a  complete timelike limit line $\sigma$. If we prove that $\sigma$ is $\gamma$-adapted we are done: the local structural result Proposition \ref{Proposition: higherregularity, weightconstant} then implies that $\sigma_0$ belongs to $S$  and is the ${\sf d}$-limit of the $x_n$'s.

The continuity of $\ell$ on $\{\ell\geq 0\}$ and the analogous property of the $\sigma^n$'s give that $\ell(\sigma_s,\sigma_t)=t-s$ for any $s,t\in\R$, $s<t$, thus recalling that $\b^+\geq \b^-$ on $I(\gamma)$, to conclude it suffices to prove that $\b^+(\sigma_T)\leq T\leq \b^-(\sigma_T)$ for every $T\in \R$. Fix $T$, recall that for $t\geq R+T$ we have $\sigma_T,\sigma^n_T\leq \gamma_t$, $n\in\mathbf N$, and that $\sigma^n_T\to \sigma_T$ in $M$. Thus the continuity of $\ell$ in $\{\ell\geq 0\}$ yields
\[
\b_t(\sigma_T)=t-\ell(\sigma_T,\gamma_t)=\lim_{n\to\infty}t-\ell(\sigma^n_T,\gamma_t).
\]
The identities $\sigma^n_T={\sf Fl}_T(x_n)$ and $\gamma_t={\sf Fl}_t(\gamma_0)$ together with the product structure established in the previous step give the bound $\ell(\sigma^n_T,\gamma_t)\geq \sqrt{(t-T)^2-{\sf d}^2(x_n,\gamma_0)}\geq  \sqrt{(t-T)^2-R^2}$, hence
\[
\b^+(\sigma_T)=\lim_{t\uparrow+\infty}\b_t(\sigma_T)\leq \lim_{t\uparrow+\infty}t-\sqrt{(t-T)^2-R^2}=T.
\]
An analogous argument holds for $\b^-$, hence the proof is complete.
\end{proof}

It is in principle not yet clear that the image of ${\sf Fl}$ in the last statement is the whole manifold $M$. {Inspired by 
Eschenburg's \cite[Proof of Lem. 7.3]{Eschenburg88}, we prove this using the following general lemma.}

\begin{lemma}[Reaching the boundary of any open set with a geodesic]
\label{le:reaching}
    Let $M$ be a connected, smooth, second countable manifold
    equipped with a  $C^1$ semi-Riemannian metric tensor $g$  and $U\subsetneq M$ open and nonempty.
    Then there is a $g$-geodesic $\sigma:[0,1]\to M$ such that $\sigma(s)\in U$ for every $s\in[0,1)$ and $\sigma(1)\notin U$. 
\end{lemma}

\begin{proof}
Let $\tilde g$ be the auxiliary smooth complete Riemannian tensor on $M$ provided using second countability by Nomizu and Ozeki \cite{NomizuOzeki61}.
Pick $z\in \partial U$ and let $R>0$ be such that for any $y\in B^{\tilde g}_R(z)$ the $\tilde g$-exponential map is a smooth diffeomorphism from  $B_R^{\tilde g}(0)\subset T_yM$ to its image in $M$. Pick $y\in B^{{\tilde g}}_R(z)\cap U$, let $r:=\min_{x\in M\setminus U}{\sf d}_{\tilde g}(y,x)>0$ and pick $x$ achieving the minimum. Then the (only) $\tilde g$-geodesic from $x$ to $y$ meets the smooth boundary of $B^{{\tilde g}}_r(y)$ transversally. Let $v\in T_{x}M$ be the speed at 0 of such geodesic. Transversality implies that if  $\eta:[0,1]\to M$ is a  $C^1$ curve with $\eta(0)=x$ and $\eta'(0)=v$, then for some $\varepsilon>0$ we have $\eta((0,\varepsilon))\subset B^{{\tilde g}}_r(y)\subset U$. In particular, this applies to the solution of the $g$-geodesic equation with such initial data provided by Lemma \ref{le:C1precomp}: reversing time and rescaling we get the desired curve.
\end{proof}

We can now conclude:
\begin{proof}[Proof of Theorem \ref{T2:weighted splitting}]
With the notation of Theorem \ref{T:productstructure}, let $U\subset M$ be the open image of $\R\times S$ in $M$ via ${\sf Fl}$. We want to prove that  $U=M$. If not, we  use Lemma \ref{le:reaching} above to find a $g$-geodesic  $\sigma:[0,1]\to M$ with $\sigma_t\in U$ for $t<1$ and $\sigma_1\notin U$. For $t<1$ put $(\sigma^\R_t,\sigma^S_t):={\sf Fl}^{-1}(\sigma_t)\in \R\times S$ and notice that since ${\sf Fl}:\R\times S\to U\subset M$ is an isometric embedding, the  curves $\sigma^\R,\sigma^S$ are  geodesics  on $[0,1)$ in $\R,S$ respectively. Since $\R$ and $S$ are complete, these geodesics have limits $\sigma^\R_1,\sigma^S_1$ and by the continuity of ${\sf Fl}$ we have $U\ni {\sf Fl}_{\sigma^\R_1}(\sigma^S_1)=\sigma_1$, contradicting the assumption $\sigma_1\notin U$.

It remains to discuss the regularity of ${\sf Fl}$. The construction and $\b\in C^2(M)$ yield $\sf Fl \in C^1$.
The improvement to ${\sf Fl} \in C^2$ follows from  Proposition \ref{Prop: semiRiemannisometriesregularity}.
Alternately,  it follows from the Riemannian isometry results of Calabi and Hartman \cite{CalabiHartman1970} or Taylor \cite{taylor2006existence} as in Remark \ref{R:weighted splitting}.
\end{proof}

\section{Outlook}\label{Section: Conclusion}

In the research of low regularity spacetimes, 
it is natural to ask whether such splitting theorems remain true 
for (continuous) Geroch--Traschen metrics \cite{GerochTraschen87} (see also \cite{steinbauer2009geroch}), which are the lowest regularity class of metrics for which curvature tensors can be stably defined distributionally. However, as far as splitting theorems are concerned, that regularity seems to be outside of the scope of current methods. More reasonably, one could investigate locally Lipschitz continuous metrics, for which promising approximation results were recently developed \cite{calisti2025hawking}.

In analogy with splitting theorems for low regularity metrics, there is the problem of 
establishing a splitting theorem for infinitesimally Minkowskian metric measure spacetimes satisfying timelike curvature-dimension conditions first mentioned with the octet \cite{BeranOctet24+} and more explicitly stated as an open problem in Braun's survey \cite{braun2025new}. Such a result in the metric spacetime setting can be expected to give rise to a structure theory for spacetimes with lower timelike Ricci curvature bounds (and upper dimension bounds), as 
{Gigli's nonsmooth extension~\cite{Gigli13+, Gigli14} of the Cheeger-Gromoll splitting theorem \cite{CheegerGromoll71} to $\mathsf{RCD}(0,N)$ spaces 
did} in Riemannian signature \cite{MondinoNaber19}.

\appendix

\section{Regularity of semi-Riemannian isometries}

In this appendix we prove the following proposition concerning the regularity of isometries between semi-Riemannian manifolds equipped with $C^1$ metric tensors.
In a corollary we combine it with a result of Hartman~\cite{Hartman58L} to yield the corresponding improvement when $g \in C^k$ for $1<k \in \mathbf N$. 
When the metric has Riemannian signature the same results were among those obtained by Calabi--Hartman~\cite{CalabiHartman1970}, and extended to metrics of $C^{k,\alpha}$ regularity
for $k \in \mathbf N$ and $0<\alpha<1$ by Taylor \cite{taylor2006existence} using different methods.

\begin{proposition}[Regularity of semi-Riemannian isometries]\label{Prop: semiRiemannisometriesregularity}
Let $(M,g)$ and $(\bar{M},\bar{g})$ be semi-Riemannian manifolds (of equal signature and dimension) such that $g$ and $\bar{g}$ are of regularity $C^1$. Suppose $\Phi:M \to \bar{M}$ is a $C^1$-diffeomorphism satisfying $\Phi^*\bar{g} = g$. Then $\Phi$ is of regularity $C^{2}$.
\end{proposition}

\begin{proof}
At any fixed point $x_0 \in M_1$ a smooth change of variables allows us to find a smooth coordinate system
in which $x_0=0$ and the Christoffel symbols \eqref{Christoffel} vanish at $x_0$;
we may also assume $g_{ij}(x_0)$ has unit magnitude whenever $i=j$ and vanishes for all $i\ne j$.
We call such coordinates {\em infinitesimally normal} at $x_0$;
(they differ from normal coordinates in that we do not assert the geodesy of any segment through the origin). 
The same is true of $\bar g$ at the point $\bar x_0=\Phi(x_0)$.  We henceforth express both metrics in such coordinates and regard $\Phi$ as a map from $\R^n$ to $\R^n$ which acts as a $C^1$ diffeomorphism in a neighbourhood of $0=\Phi(0)$, whose derivative is denoted by $\p\Phi(x)$.  For any $y,v \in B^{e}_r(0)$ in a small enough coordinate ball,  there exists a (possibly nonunique)
$g$-geodesic $\{y_t\}_{t\in[0,1]}\subset\R^n$
with $y_0=x$ and $\dot y_0=v$. 
Since $\Gamma_{jk}^i(y_t) = o(1)$ as $r \to 0$, for the Euclidean distance $|\cdot |_{e}$ on our infinitesimally normal coordinates, the geodesic equation yields
$$|\dot y_1-\dot y_0|_{e} \le c |v|_{e}^2.$$
It is easy to see that $\Phi$ must map $g$-geodesics to $\bar{g}$-geodesics: Indeed, if $y_t$ is a $g$-geodesic, then $\Phi(y_t)$ is $C^1$ and a weak solution of the $\bar{g}$-geodesic equation. Since $\bar{g}$ is $C^1$ and the right-hand side of the geodesic equation is $-(\Gamma^{\bar{g}})^{i}_{jk}\frac{d}{dt}\Phi(y_t)^j \frac{d}{dt} \Phi(y_t)^k$ and thus continuous, we conclude that $\frac{d}{dt} \Phi(y_t)^i$ is $W^{1,1}_{loc}$, in particular it is absolutely continuous and thus a.e.\ differentiable. So, in fact, $\Phi(y_t)$ solves the $\bar{g}$-geodesic equation a.e.\ pointwise. Again invoking the continuity of the right-hand side, we conclude that $\Phi(y_t)^k$ is $C^2$ and a classical solution.

Hence, since $\bar y_t := \Phi(y_t)$ is a $\bar{g}$-geodesic and 
$\Phi$ is a diffeomorphism,
$$|\dot {\bar y}_1-\dot {\bar y}_0|_{e} \le c |v|_{e}^2$$
similarly.  In both cases $\lim_{r\to 0} c= 0$.
Since $\dot {\bar y}_t= \p\Phi(y_t) \dot y_t$ and $v=\dot y_0$, the foregoing imply
\begin{align*}
|(\p\Phi(y_1)-\p \Phi(y_0))v|_{e} 
&\le |\p\Phi(y_1)\dot y_1 - \p\Phi(y_0)\dot y_0|_{e} + c|v|_{e}^2
\\&\le c|v|_{e}^2.
\end{align*}
Arbitrariness of $y_0,v \in B^{e}_r(0)$ therefore yield
\begin{align*}
\|\p\Phi(y_1)-\p \Phi(y_0)\|_{\R^n \times \R^n} \le c|v|_e
\le c|y_1-y_0|_e.
\end{align*}
Thus $\p\Phi$ has a Lipschitz constant $c$ which decays to $0$ as $r \to 0$.  This shows $\Phi$ is twice differentiable at the origin and its derivative vanishes there.  Since $x_0$ was arbitrary, $\Phi$ is twice differentiable at every other point too; thus $\p^2 \Phi(y_0)$ exists.  Now
$\|\p^2 \Phi(y_0)\| \le c$, so $\lim_{r\to 0} c=0$ shows $\Phi$ to be twice continuously differentiable at the origin --- hence at every other point too. 
\end{proof}

\begin{corollary}[Higher regularity of semi-Riemannian isometries]\label{C: semiRiemannisometriesregularity}
With the same hypotheses and notation as Proposition \ref{Prop: semiRiemannisometriesregularity}, if $g$ and $\bar{g}$ are of regularity $C^k$
for any integer $k \ge 1$ then $\Phi$ is of regularity $C^{k+1}$.
\end{corollary}

\begin{proof}
Follows immediately by combining Proposition \ref{Prop: semiRiemannisometriesregularity} with
Hartman \cite[Thm. II]{Hartman58L}.
\end{proof}

\section*{Acknowledgments}

MB is supported by the EPFL through a Bernoulli Instructorship. NG is supported in part by the MUR PRIN-2022JJ8KER grant ``Contemporary perspectives on geometry and gravity". RJM's research is supported in part by the Canada Research Chairs program CRC-2020-00289, a grant from the Simons Foundation (923125, McCann), Natural Sciences and Engineering Research Council of Canada Discovery Grant RGPIN- 2020--04162, and Toronto's Fields Institute for the Mathematical Sciences, where this collaboration began. AO was supported in part by the Canada Research Chairs program CRC-2020-00289 and Natural Sciences and Engineering Research Council of Canada Grant RGPIN-2020-04162.

The authors are grateful to Guido De Philippis and Cale Rankin for stimulating exchanges.

This research was funded in part by the Austrian Science Fund (FWF) [Grants DOI \href{https://doi.org/10.55776/STA32}{10.55776/STA32}, \href{https://doi.org/10.55776/EFP6}{10.55776/EFP6} and \href{https://doi.org/10.55776/J4913}{10.55776/J4913}]. For open access purposes, the authors have applied a CC BY public copyright license to any author accepted manuscript version arising from this submission. 

\bibliographystyle{plain}%{mf}
\bibliography{newbib.bib}

\end{document}